\documentclass[11pt]{article}
\usepackage{mathtools}
\usepackage{amsmath,amsthm}
\usepackage{csquotes}
\usepackage{amssymb}
\usepackage{graphicx}
\usepackage{subfig}
\usepackage{mathpazo}
\usepackage{soul}  
\usepackage{xcolor}
\usepackage{empheq}
\usepackage{lipsum}

\definecolor{myblue}{rgb}{.8, .8, 1}

\usepackage[obeyspaces,hyphens,spaces]{url}
\usepackage[unicode]{hyperref}
\hypersetup{
    colorlinks=true,
    linkcolor=blue,
    citecolor=blue,
    filecolor=magenta,
    urlcolor=cyan
}

\usepackage[
  open,
  openlevel=2,
  atend,
  numbered
]{bookmark}

\usepackage[capitalize, nameinlink, noabbrev]{cleveref}
\crefname{equation}{}{}
\crefname{chapter}{Chapter}{Chapters}
\crefname{item}{item}{items}
\crefname{figure}{Figure}{Figures}
\crefname{theorem}{Theorem}{Theorems}
\crefname{lemma}{Lemma}{Lemmas}
\crefname{proposition}{Proposition}{Propositions}
\crefname{corollary}{Corollary}{Corollarys}
\crefname{definition}{Definition}{Definitions}
\crefname{fact}{Fact}{Facts}
\crefname{example}{Example}{Examples}
\crefname{algorithm}{Algorithm}{Algorithms}
\crefname{remark}{Remark}{Remarks}
\crefname{note}{Note}{Notes}
\crefname{notation}{Notation}{Notations}
\crefname{case}{Case}{Cases}
\crefname{exercise}{Exercise}{Exercises}
\crefname{question}{Question}{Questions}
\crefname{claim}{Claim}{Claims}
\crefname{enumi}{}{}

\usepackage[top= 2cm, bottom = 2 cm, left = 2.2 cm, right= 2.2 cm]{geometry}  
\numberwithin{equation}{section}

\theoremstyle{plain}
\newtheorem{theorem}{Theorem}[section]
\newtheorem{corollary}[theorem]{Corollary}
\newtheorem{fact}[theorem]{Fact}
\newtheorem{lemma}[theorem]{Lemma}
\newtheorem{proposition}[theorem]{Proposition}

\theoremstyle{definition}
\newtheorem{definition}[theorem]{Definition}
\newtheorem{example}[theorem]{Example}

\newtheorem{question}[theorem]{Question}
\newtheorem{remark}[theorem]{Remark}

\usepackage{enumitem}

\newcommand{\zer}{\ensuremath{\operatorname{zer}}}

\newcommand{\weakly}{\ensuremath{{\;\operatorname{\rightharpoonup}\;}}}

\newcommand{\dom}{\ensuremath{\operatorname{dom}}}
\newcommand{\gra}{\ensuremath{\operatorname{gra}}}
\newcommand{\Fix}{\ensuremath{\operatorname{Fix}}}
\newcommand{\Id}{\ensuremath{\operatorname{Id}}}

\newcommand{\dist}{\ensuremath{\operatorname{d}}}
\newcommand{\Pro}{\ensuremath{\operatorname{P}}}

\newcommand{\J}{\ensuremath{\operatorname{J}}}

\newcommand{\Range}{\ensuremath{\operatorname{ran}}}

\newcommand{\subreg}{\ensuremath{\operatorname{subreg}}}

\providecommand{\abs}[1]{\left|#1\right|}
\providecommand{\norm}[1]{\left\lVert#1\right\rVert}
\providecommand{\innp}[1]{\left\langle#1\right\rangle}
\providecommand{\lr}[1]{\left(#1\right)}

\begin{document}

\title{Linear Convergence of Generalized Proximal Point Algorithms for Monotone Inclusion Problems}

\author{
         Hui\ Ouyang\thanks{
                 Mathematics, University of British Columbia, Kelowna, B.C.\ V1V~1V7, Canada.
                 E-mail: \href{mailto:hui.ouyang@alumni.ubc.ca}{\texttt{hui.ouyang@alumni.ubc.ca}}.}
                 }

\date{March 25, 2022}

\maketitle

\begin{abstract}
	\noindent
We focus on the linear convergence of generalized proximal point algorithms for solving monotone inclusion problems. Under the assumption that the associated monotone operator is metrically subregular or that the inverse of the monotone operator is Lipschitz continuous, we provide $Q$-linear and $R$-linear convergence results on generalized proximal point algorithms. Comparisons between our results and related ones in the literature are presented in remarks of this work.
\end{abstract}

{\small

}

%%%%%%%%%%%%%%%%%%%%%%%%%%%%%%%%%%%%%%%%%%%%%%%%%%%%
%%%%%%%%%%%%%%%%\section{Introduction}%%%%%%%%%%%%%%%%%%%%%
%%%%%%%%%%%%%%%%%%%%%%%%%%%%%%%%%%%%%%%%%%%%%%%%%%%%
\section{Introduction} \label{sec:Introduction}
 Throughout this work,  
 \begin{align*}
 \text{$\mathcal{H}$ is a real Hilbert space},
 \end{align*}
 with inner product $\innp{\cdot,\cdot}$ and induced norm $\norm{\cdot}$.   

Let $A :\mathcal{H} \to 2^{\mathcal{H}}$ be  maximally monotone with $\zer A \neq \varnothing$. Denote the set of all nonnegative integers by $\mathbb{N} :=\{0,1,2,\ldots\}$. 
The iteration sequence of the  \emph{generalized proximal point algorithm} is generated by conforming to the iteration scheme: 
\begin{align*} 
(\forall k \in \mathbb{N}) \quad x_{k+1} = \lr{1-\lambda_{k}}x_{k} +\lambda_{k} \J_{c_{k} A}x_{k} +\eta_{k}e_{k}, 
\end{align*}
where  $x_{0} \in \mathcal{H}$ is the \emph{initial point} and $(\forall k \in \mathbb{N})$ $\lambda_{k} \in \left[0,2\right]$ and $\eta_{k} \in \mathbb{R}_{+}$ are the \emph{relaxation coefficients},  $c_{k} \in \mathbb{R}_{++}$ is the \emph{regularization coefficient},  and $e_{k} \in \mathcal{H}$ is the \emph{error term}.   

\emph{The goal of this work is to investigate the linear convergence of the generalized proximal point algorithm for solving the associated monotone inclusion problem, that is,  finding a point $\bar{x}$ in $\zer A :=  \left\{ x \in \mathcal{H} ~:~ 0 \in Ax \right\}$.}

Many celebrated optimization algorithms are actually specific cases of the generalized proximal point algorithm when the operator $A$ is specified accordingly; such algorithms include the projected gradient method \cite{Polyak1987}, the extragradient method \cite{Korpelevic1976}, the forward-backward  splitting algorithm \cite{LionsMercier1979}, the Peaceman-Rachford splitting algorithm \cite{PeacemanRachford1955},
the Douglas-Rachford splitting algorithm \cite{DouglasRachford1956,LionsMercier1979,PowellDRS1969}, the split inexact Uzawa method \cite{ZhangBurgerBressonOsher2010}, and so on; in addition, the augmented Lagrangian method (i.e., the method of multipliers) \cite{HestenesMultiplier1969} and the alternating direction method of multipliers \cite{GabayMercier1976} are instances of the generalized proximal point algorithm applied to dual problems (see, e.g., \cite{CormanYuan2014,EcksteinBertsekas1992,KimAcceleration2021,TaoYuan2018} for exposition).
Hence, studying the linear convergence of the generalized proximal point algorithm helps us deduce corresponding results on the linear convergence of algorithms mentioned above.

Main results in this work are summarized as follows. 
\begin{itemize}
	
		\item[\textbf{R1:}] We show  $R$-linear convergence results  on
	generalized proximal point algorithms 
	in  \cref{theorem:exactPAlinear} and \cref{theorem:MetriSubregLinearNosingleton} under the assumption of metrical subregularity.		
		
	\item[\textbf{R2:}] $Q$-linear convergence results of generalized proximal point algorithms are presented in  \cref{theorem:JckAMetricSub} and \cref{prop:JckAMetricSub} under the assumption of metrical subregularity.
	
	\item[\textbf{R3:}] In \cref{theorem:JckALipschitzLinear} and \cref{prop:JckALipschitzLinear}, under the assumption of the Lipschitz continuity of  the inverse of the related monotone operator, we obtain $Q$-linear convergence results of generalized proximal point algorithms. 

\end{itemize}

The rest of this work is organized as follows. We collect some basic definitions,  fundamental facts, and auxiliary results in \cref{sec:Preliminaries}. To facilitate  proofs in subsequent sections, we work on the metrical subregularity of set-valued operators in \cref{section:MetricalSubregularity}.
 In \cref{sec:KMannIterations}, we consider the inexact version of the  non-stationary Krasnosel'ski\v{\i}-Mann iterations. In particular, we establish a $R$-linear convergence result on the  non-stationary Krasnosel'ski\v{\i}-Mann iterations, which will be used to deduce the corresponding result on the generalized proximal point algorithm in \cref{sec:GPPA}. Our main results on the linear convergence of generalized proximal point algorithms are presented in \cref{sec:GPPA}. In the last section \cref{section:ConclusionFutureWork}, we summarize this work and list  some possible future work.

We now turn to the notation used in this work. 
$\Id$ stands for the \emph{identity mapping}.  
Denote by  $\mathbb{R}_{+}:=\{\lambda \in \mathbb{R} ~:~ \lambda \geq 0 \}$ and $\mathbb{R}_{++}:=\{\lambda \in \mathbb{R} ~:~ \lambda >0 \}$. 
Let $\bar{x} $ be in $ \mathcal{H}$ and let $r \in \mathbb{R}_{+}$.
$B[\bar{x};r]:= \{ y\in \mathcal{H} ~:~ \norm{y-\bar{x}} \leq r \}$ is the \emph{closed ball centered at $\bar{x}$ with radius $r$}.  
Let $C$ be a nonempty set of $\mathcal{H}$. Then $(\forall x \in \mathcal{H})$ $\dist \lr{x, C} = \inf_{y \in C} \norm{x-y}$. If $C$ is   nonempty closed and convex, then the \emph{projector} (or \emph{projection operator}) onto $C$ is the operator, denoted by $\Pro_{C}$,  that maps every point in $\mathcal{H}$ to its unique projection onto $C$, that is, $(\forall x \in \mathcal{H})$ $\norm{x - \Pro_{C}x} = \dist \lr{x, C} $.  
Let $\mathcal{D} $ be a nonempty subset of $\mathcal{H}$ and let $T: \mathcal{D} \to  \mathcal{H}$. $\Fix T :=\{ x \in \mathcal{D}~:~ x = T(x) \}$ is the \emph{set of fixed points  of $T$}. 
Let $A: \mathcal{H} \to 2^{\mathcal{H}}$ be a set-valued operator. Then $A$ is characterized by its \emph{graph} $\gra A:= \{ (x,u) \in \mathcal{H} \times \mathcal{H} ~:~ u\in A(x) \}$. The \emph{inverse} of $A$, denoted by $A^{-1}$, is defined through its graph $\gra A^{-1} :=  \{ (u,x) \in \mathcal{H} \times \mathcal{H} ~:~ (x,u) \in \gra A \}$.
The  \emph{domain},  \emph{range}, and \emph{set of zeros}  of $A$   are  defined by $\dom A := \left\{ x \in \mathcal{H} ~:~ Ax \neq \varnothing \right\}$,  $\Range A := \left\{ y\in \mathcal{H} ~:~ \exists~ x \in \mathcal{H} \text{ s.t. } y\in Ax \right\}$, and $\zer A :=  \left\{ x \in \mathcal{H} ~:~ 0 \in Ax \right\}$, respectively. 
Let $(y_{k})_{k \in \mathbb{N}}$ be a sequence in $\mathcal{H}$ and let $\bar{y} $ be in $\mathcal{H}$.  $\Omega \lr{ \lr{y_{k}}_{k \in \mathbb{N}} }$ stands for the \emph{set of all weak sequential clusters of the sequence $(y_{k})_{k \in \mathbb{N}}$}.  If  $(y_{k})_{k \in \mathbb{N}}$ \emph{converges} (\emph{strongly}) to $\bar{y}$, then we denote by $y_{k} \to \bar{y}$.  $(y_{k})_{k \in \mathbb{N}}$ \emph{converges weakly} to  $\bar{y} $ if, for every $u \in \mathcal{H}$, $\innp{y_{k},u} \rightarrow \innp{y,u}$; in symbols, $y_{k} \weakly \bar{y}$. Suppose that $(y_{k})_{k \in \mathbb{N}}$ converges to  $\bar{y}$. Then 
$(y_{k})_{k \in \mathbb{N}}$  is \emph{$R$-linearly convergent}  (or \emph{converges $R$-linearly}) to $\bar{y}$ if $\limsup_{k \to \infty} \lr{\norm{y_{k} -\bar{y}}}^{\frac{1}{k}} <1$; when $(\forall k \in \mathbb{N})$ $y_{k} \neq \bar{y}$, we say $(y_{k})_{k \in \mathbb{N}}$  is \emph{$Q$-linearly convergent}  (or \emph{converges $Q$-linearly}) to $\bar{y}$ if $\limsup_{k \to \infty} \frac{\norm{y_{k+1} -\bar{y}}}{\norm{y_{k} -\bar{y}}} <1$.
For other notation not explicitly defined here, we refer the reader to \cite{BC2017}.

%%%%%%%%%%%%%%%%%%%%%%%%%%%%%%%%%%%%%%%%%%%%%%%%%%%%
%%%%%%%%%%%%%%%%\section{Preliminaries}%%%%%%%%%%%%%%%%%%%%%%
%%%%%%%%%%%%%%%%%%%%%%%%%%%%%%%%%%%%%%%%%%%%%%%%%%%%
\section{Preliminaries} \label{sec:Preliminaries}

 The definitions, facts, and lemmas gathered in this section are fundamental to our analysis in the subsequent sections. 
\subsection{Nonexpansive operators}

All algorithms considered in this work are based on nonexpansive operators. 
 
\begin{definition} {\rm \cite[Definition~4.1]{BC2017}} \label{definition:Nonexpansive}
	Let $D$ be a nonempty subset of $\mathcal{H}$ and let $T: D \to \mathcal{H}$. Then $T$ is 
	\begin{enumerate}
		\item  \label{definition:Nonexpansive:nonexp} \index{nonexpansive operator} \emph{nonexpansive} if it is Lipschitz continuous with constant $1$, i.e., $(\forall  x \in D)$ $ (\forall y \in D) $ $\norm{Tx-Ty} \leq \norm{x-y}$;

		\item  \label{definition:Nonexpansive:firmlynonexp} \index{firmly nonexpansive operator} 
		\emph{firmly nonexpansive} if $	(\forall  x \in D) $ $(\forall y \in D) $ $ \norm{Tx-Ty}^{2} +\norm{(\Id -T)x - (\Id -T)y}^{2} \leq \norm{x-y}^{2}$.
	\end{enumerate}
\end{definition}

Although \cite[Definition~4.33]{BC2017} considers only the case $\alpha \in \left]0,1\right[\,$, we extend the definition to $\alpha \in \left]0,1\right]$. Clearly, by \cref{definition:Nonexpansive}\cref{definition:Nonexpansive:nonexp} and \cref{definition:averaged}, $T$ is $1$-averaged if and only if $T$ is nonexpansive. This extension will facilitate our future statements. It is clear that both firmly nonexpansive and averaged operators must be nonexpansive.
\begin{definition} {\rm \cite[Definition~4.33]{BC2017}} 	\label{definition:averaged}
	\index{averaged nonexpansive operator} 
	Let $D$ be a nonempty subset of $\mathcal{H}$,  let $T: D \to \mathcal{H}$ be nonexpansive, and let $\alpha \in \left]0,1\right]$. Then $T$ is \emph{averaged with constant $\alpha$}, or  \emph{$\alpha$-averaged}, if there exists a nonexpansive operator $R: D \to \mathcal{H}$ such that $T=(1-\alpha) \Id +\alpha R$.	
\end{definition}

\begin{fact} \label{fact:lemma:yxzx} {\rm \cite[Proposition~2.7(ii)]{OuyangStabilityKMIterations2022}}
	Let $\alpha \in \left]0,1\right]$ and let $T:\mathcal{H} \to \mathcal{H}$ be an $\alpha$-averaged operator with $\Fix T \neq \varnothing$. Let $  x $ and $e $ be in $ \mathcal{H}$ and  let $\lambda  $
	and  $\eta $ be in $\mathbb{R}_{+}$. Define
	\begin{align*}
	y_{x} := (1-\lambda) x+\lambda Tx \quad \text{and} \quad 
	z_{x} :=   (1-\lambda) x+\lambda Tx +\eta e=y_{x} +\eta e.
	\end{align*}
	Then for every $\bar{x} \in \Fix T$,
	\begin{subequations} \label{eq:fact:lemma:yxzx} 
		\begin{align}
	&\norm{y_{x} -\bar{x}}^{2} \leq \norm{x -\bar{x}}^{2} -\lambda \lr{\frac{1}{\alpha} -\lambda} \norm{x -Tx}^{2};  \label{eq:fact:lemma:yxzx:y} \\
	& \norm{z_{x} -\bar{x}}^{2}  
	\leq  \norm{x -\bar{x}}^{2} -\lambda  \lr{\frac{1}{\alpha} -\lambda} \norm{x -Tx}^{2} +\eta \norm{e} \lr{2\norm{y_{x} -\bar{x}} +\eta \norm{e}}.  \label{eq:fact:lemma:yxzx:z} 
	\end{align}	
	\end{subequations}
\end{fact}		

		\begin{lemma} \label{lemma:yxzx}
			Let $\alpha \in \left]0,1\right]$ and let $T:\mathcal{H} \to \mathcal{H}$ be an $\alpha$-averaged operator with $\Fix T \neq \varnothing$. Let $x$ and $e$ be in $  \mathcal{H}$, let $\lambda \in \mathbb{R}$,
			and  let $\eta \in \mathbb{R}_{+}$. Define
			\begin{align*}
			y_{x} := (1-\lambda) x+\lambda Tx \quad \text{and} \quad 
			z_{x} :=   (1-\lambda) x+\lambda Tx +\eta e=y_{x} +\eta e.
			\end{align*}
			Let $\varepsilon$ and $ \beta $ be in $  \mathbb{R}_{+}$ with $\eta \varepsilon \in \left[0,1\right[$ and let $ \bar{x} \in \mathcal{H}$. Suppose   that $\norm{e} \leq \varepsilon \norm{x -z_{x}}$ and $\norm{y_{x} -\bar{x}} \leq \beta \norm{x -\bar{x}}$. Then
			\begin{align*}
			\norm{z_{x} -\bar{x}} \leq \frac{\beta + \eta \varepsilon}{1 - \eta \varepsilon} \norm{x -\bar{x}}.
			\end{align*}		
\end{lemma}

\begin{proof}
	Apply  the assumptions $\norm{e} \leq \varepsilon \norm{x -z_{x}}$ and $\norm{y_{x} -\bar{x}} \leq \beta \norm{x -\bar{x}}$ in the following second inequality to force that  
	\begin{align*}
	\norm{z_{x} -\bar{x}}   
	\leq  \norm{ y_{x}-\bar{x} } + \eta  \norm{e } 
	\leq   \beta \norm{x -\bar{x}} + \eta\varepsilon \norm{x -z_{x}} 
	\leq   \beta \norm{x -\bar{x}} + \eta\varepsilon \lr{ \norm{x -\bar{x}} +\norm{\bar{x} -z_{x}}},
	\end{align*}
	which implies directly that $\norm{z_{x} -\bar{x}} \leq \frac{\beta + \eta \varepsilon}{1 - \eta \varepsilon} \norm{x -\bar{x}}$.
\end{proof}

\subsection{Resolvent of monotone operators}

\begin{definition} \label{definition:monotone}
	Let $ A : \mathcal{H} \to 2^{\mathcal{H}}$ be a set-valued operator. Then we say 
	\begin{enumerate}
		\item  {\rm \cite[Definition~20.1]{BC2017}} $A$ is  \emph{monotone}  if $	(\forall (x,u) \in \gra A ) $ $ (\forall  (y,v) \in \gra A) $ $ \innp{x-y, u-v} \geq 0$;
		\item  {\rm \cite[Definition~20.20]{BC2017}}  a monotone operator $A$ is \emph{maximally monotone} (or \emph{maximal monotone}) if there exists no monotone operator $ B :\mathcal{H} \to 2^{\mathcal{H}}$ such that $\gra B$ properly contains $\gra A$, i.e., for every $(x,u) \in \mathcal{H} \times \mathcal{H}$, 
		\begin{align*}
		(x,u) \in \gra A \Leftrightarrow \lr{ \forall  (y,v) \in \gra A } \innp{x-y, u-v} \geq 0.
		\end{align*}
	\end{enumerate} 
\end{definition}

\begin{definition} {\rm \cite[Definition~23.1]{BC2017}} \label{defn:ResolventApproxi}
	Let $A: \mathcal{H} \to 2^{\mathcal{H}}$ and let $\gamma \in \mathbb{R}_{++}$. The \emph{resolvent of $A$} \index{resolvent} is 
	\begin{align*}
	\J_{A} = (\Id + A)^{-1}.
	\end{align*}
\end{definition}

 \cref{fact:cAMaximallymonotone} illustrates that the resolvent of a maximally monotone operator is single-valued, full domain, and firmly nonexpansive, which is essential to our study on generalized proximal point algorithms in subsequent sections. 

\begin{fact}  {\rm \cite[Proposition~23.10]{BC2017}} \label{fact:cAMaximallymonotone}
	Let  $A: \mathcal{H} \to 2^{\mathcal{H}}$ be such that $\dom A \neq \varnothing$, set $D:= \Range A$, and set $T=\J_{A } |_{ D}$. Then $A$ is maximally monotone if and only if $T$ is firmly nonexpansive and $D =\mathcal{H}$.
\end{fact}

 \begin{lemma} {\rm \cite[Proposition~23.38]{BC2017}} \label{cor:fact:FixJcAzerA}
 	Let $A: \mathcal{H} \to 2^{\mathcal{H}}$ be maximally monotone and  let $ \gamma \in \mathbb{R}_{++}$. Then $ \J_{\gamma A}$ is $\frac{1}{2}$-averaged and 
 	\begin{align} \label{eq:fact:FixJcAzerA}
 \zer \lr{\Id -\J_{\gamma A} }  =	\Fix \J_{\gamma A} = \zer A.
 	\end{align} 
 \end{lemma}

 \begin{proof}
 Because $A$ is maximally monotone, due to \cite[Proposition~20.22]{BC2017}, \cref{fact:cAMaximallymonotone},	and \cite[Remark~4.34(iii)]{BC2017}, we know that $\gamma A$ is also maximally monotone and that $ \J_{\gamma A}$ is $\frac{1}{2}$-averaged, which, combined with \cite[Proposition~23.38]{BC2017}, yields \cref{eq:fact:FixJcAzerA}.
 \end{proof}

\begin{fact} {\rm \cite[Proposition~23.2(ii)]{BC2017}} \label{corollary:JcA}
	Let  $A: \mathcal{H} \to 2^{\mathcal{H}}$ be maximally monotone. Then 
	\begin{align*}
\lr{x \in \mathcal{H}} \lr{\gamma \in \mathbb{R}_{++}} \quad 	\lr{ \J_{\gamma A} x , \frac{1}{\gamma} \lr{x-\J_{\gamma A} x}}   \in \gra A. 
	\end{align*}
\end{fact}

\begin{fact} \label{lemma:resolvents:yxzx} {\rm \cite[Lemma~2.12]{OuyangStabilityKMIterations2022}}
	Let $A: \mathcal{H} \to 2^{\mathcal{H}}$ be maximally monotone with $\zer A \neq \varnothing$.    Let  $ x $ and $e$  be in $ \mathcal{H}$, let $\lambda $ and $\eta$ be in $ \mathbb{R}_{+}$, and let 	$\gamma \in \mathbb{R}_{++}$. Define
	\begin{align*}
	 y_{x} := (1-\lambda) x+\lambda \J_{\gamma A}x \quad \text{and} \quad z_{x} :=   (1-\lambda) x+\lambda \J_{\gamma A}x +\eta e. 
	\end{align*}
	Then  $\lr{ \forall \bar{x} \in \zer A}$ $\norm{y_{x} -\bar{x}}^{2} \leq \norm{x -\bar{x}}^{2} -\lambda \lr{2 -\lambda} \norm{x -\J_{\gamma A}x}^{2}$.
\end{fact}

\begin{fact}\label{lemma:JGammaAFix}{\rm \cite[Lemma~2.15]{OuyangStabilityKMIterations2022}}
	Let $A: \mathcal{H} \to 2^{\mathcal{H}}$ be maximally monotone with $\zer A \neq \varnothing$ and  let $ \gamma \in \mathbb{R}_{++}$. Then $	(\forall x \in \mathcal{H})	(\forall z \in \zer A) $ $\norm{\J_{\gamma A} x -z}^{2} +\norm{ \lr{\Id - \J_{\gamma A}}x}^{2} \leq \norm{x -z}^{2}$.
\end{fact}
\subsection{Miscellany}
 
Results presented in this subsection will facilitate some proofs later.

The identity shown in \cref{fact:lambdat2} below is essentially given on \cite[Page~4]{GuYang2019} to illustrate the linear convergence of a special instance of the exact version of the generalized proximal point algorithm. 
\begin{fact} {\rm \cite[Page~4]{GuYang2019}} \label{fact:lambdat2}
	Let $ t$ be in $\mathbb{R} \smallsetminus \{-1\}$ and let $\lambda$ be in $ \mathbb{R}$.  Then 
	\begin{align*}
	\lr{1-\frac{\lambda}{t+1}}^{2}  -  \lr{ 1 - \lambda \lr{2-\lambda} \frac{1}{1+t^{2}} } = \frac{ 2t\lambda }{  \lr{1+t^{2}} (t+1)^{2}} \lr{1-\lambda -t^{2}}.
	\end{align*}
\end{fact}

The following result will be used to compare  convergence rates of generalized proximal point algorithms later. 
\begin{corollary} \label{corollary:lambda=1}
	Let $ t$ be in $\mathbb{R}_{+} $. Then $	\lr{1-\frac{1}{t+1}}^{2}  \leq  \lr{ 1 -  \frac{1}{1+t^{2}} }$.
\end{corollary}

\begin{proof}
	Apply \cref{fact:lambdat2} with $\lambda =1$ to deduce that 
	\begin{align*}
	\lr{1-\frac{1}{t+1}}^{2}  -  \lr{ 1 - \frac{1}{1+t^{2}} }  = \frac{ 2t  }{  \lr{1+t^{2}} (t+1)^{2}} \lr{1-1 -t^{2}}  
	= -t^{2} \frac{ 2t  }{  \lr{1+t^{2}} (t+1)^{2}} \leq 0,
	\end{align*}
	which leads to the required inequality. 
\end{proof}

\cref{lemma:uvINEQ} will play a critical role in the proof of \cref{theorem:MetrSubregOptimalBounds}\cref{theorem:MetrSubregOptimalBounds:leq1}   later.
\begin{lemma} \label{lemma:uvINEQ}
	Let $u$ and $v$ be in $\mathcal{H} $ and let $t \in \mathbb{R}_{++}$ such that 
	\begin{align} \label{eq:lemma:uvINEQ}
	\norm{v} \leq t \norm{u-v}.
	\end{align}
	Then the following hold. 
	\begin{enumerate}
		\item \label{lemma:uvINEQ:INEQ} $\lr{\forall \lambda \in \left[0,1\right] }$ $\norm{(1-\lambda)u +\lambda v}^{2} 
		\leq   \lr{1-\frac{\lambda}{t+1}}^{2} \norm{u}^{2} +\lambda(t^{2}+\lambda-1) \norm{\frac{\sqrt{t}}{1+t}u-\frac{1}{\sqrt{t}}v}^{2}$.

		Moreover, the equality holds     if and only if $\norm{v} = t \norm{u-v}$.

		\item \label{lemma:uvINEQ:RestrictLambda} Suppose that  $0 \leq \lambda \leq 1-t^{2}$. Then
		\begin{align*}
		\norm{(1-\lambda)u +\lambda v}^{2} \leq \lr{1-\frac{\lambda}{t+1}}^{2} \norm{u}^{2} \leq \norm{u}^{2}.
		\end{align*}
	\end{enumerate} 
\end{lemma}

\begin{proof}
	\cref{lemma:uvINEQ:INEQ}:
	Let $\lambda \in \left[0,1\right]$. Set $\zeta:= \frac{\lambda (1-\lambda)}{t} + \frac{\lambda^{2}}{1+t}$.  Since $\lambda \in \left[0,1\right]$ and $t \in \mathbb{R}_{++}$, we know that $\zeta \geq 0$. This combined with \cref{eq:lemma:uvINEQ} guarantees that
	\begin{subequations}  \label{eq:lemma:uvINEQ:norm}
		\begin{align}
		&\norm{(1-\lambda)u + \lambda v}^{2} \\
		\leq  & \norm{(1-\lambda)u +\lambda v}^{2} +\zeta \left(   t^{2} \norm{u-v}^{2} -\norm{v}^{2}  \right)\\
		=  	&(1-\lambda)^{2} \norm{u}^{2} +  \lambda^{2} \norm{v}^{2}  
		+ 2 \lambda (1-\lambda) \innp{u,v} +
		\zeta t^{2} \left( \norm{u}^{2} -2\innp{u,v} +\norm{v}^{2} \right) -\zeta \norm{v}^{2} \\
		= 	& \lr{1-\frac{\lambda}{t+1}}^{2} \norm{u}^{2} + \left( (1-\lambda)^{2} -  \lr{1-\frac{\lambda}{t+1}}^{2} +\zeta t^{2}  \right) \norm{u}^{2} 
		+2 \left( \lambda (1-\lambda) -\zeta t^{2} \right) \innp{u,v}  \notag  \\
		&+ \left( \lambda^{2} +\zeta (t^{2}-1) \right) \norm{v}^{2}.
		\end{align}
	\end{subequations}	 
	On the other hand, by some elementary algebra, it is easy to establish that 
	\begin{subequations}  \label{eq:lemma:uvINEQ:lambda}
		\begin{align}
		&(1-\lambda)^{2} -  \lr{1-\frac{\lambda}{t+1}}^{2} +\zeta t^{2} = \frac{t\lambda}{(1+t)^{2}}\lr{t^{2}+\lambda-1};\\
		& \lambda (1-\lambda) -\zeta t^{2} = - \frac{\lambda}{1+t} \lr{t^{2}+\lambda -1};\\
		& \lambda^{2} +\zeta (t^{2}-1)  =\frac{\lambda}{t} \lr{t^{2}+\lambda-1}.
		\end{align}	
	\end{subequations}
	Combine \cref{eq:lemma:uvINEQ:norm} and \cref{eq:lemma:uvINEQ:lambda} to ensure that
	\begin{align*}
	\norm{(1-\lambda)u + \lambda v}^{2} 
	\leq & \lr{1-\frac{\lambda}{t+1}}^{2} \norm{u}^{2} +\lambda  \lr{t^{2}+\lambda-1} \lr{  \frac{t}{(1+t)^{2}}\norm{u}^{2} -  \frac{2}{1+t} \innp{u,v} +\frac{1}{t} \norm{v}^{2}}\\
	= & \lr{1-\frac{\lambda}{t+1}}^{2} \norm{u}^{2} +\lambda(t^{2}+\lambda-1) \norm{\frac{\sqrt{t}}{1+t}u-\frac{1}{\sqrt{t}}v}^{2}.
	\end{align*}
	In addition, based on our proof above, it is clear that the inequality above turns to an equality if and only if $\norm{v} = t \norm{u-v}$.
	
	\cref{lemma:uvINEQ:RestrictLambda}: Inasmuch as $\lambda \geq 0$ and $t > 0$, we have that
	\begin{subequations}
		\begin{align}
		&\lambda \leq 1-t^{2} \Leftrightarrow t^{2}+\lambda-1 \leq 0; \label{eq:lemma:uvINEQ:RestrictLambda:lambda}\\
		& \left( 1 -\frac{\lambda}{t+1}\right)^{2} \leq 1 \Leftrightarrow 0\leq \frac{\lambda}{t+1} \leq 2 \Leftrightarrow \lambda \leq 2 (t+1). \label{eq:lemma:uvINEQ:RestrictLambda:bracket}
		\end{align}
	\end{subequations} 
	Notice that $0 \leq \lambda \leq 1-t^{2}$ and that $\lambda \leq 1-t^{2} \leq 1 \leq 2 (t+1)$.  Hence, as a consequence of \cref{lemma:uvINEQ:INEQ}, the first and second inequalities in \cref{lemma:uvINEQ:RestrictLambda} follow directly from  \cref{eq:lemma:uvINEQ:RestrictLambda:lambda} and \cref{eq:lemma:uvINEQ:RestrictLambda:bracket}, respectively.
\end{proof}

\begin{remark}  \label{remark:uvINEQ}
	\cref{lemma:uvINEQ} is inspired by the second part of the proof of \cite[Theorem~3.1]{GuYang2019} which works on the convergence rate of a special instance of the exact version of the  generalized proximal point algorithm in $\mathbb{R}^{n}$.  We assume there is a tiny typo, $\frac{t_{k}}{t^{2}_{k} +1}$ should be $\frac{t_{k}}{(t_{k} +1)^{2}}$, in the equation $(3.10)$ of the proof of \cite[Theorem~3.1]{GuYang2019}. 
	Hence, the part after $(3.10)$ in the proof of  \cite[Theorem~3.1]{GuYang2019} could have been simplified. 
\end{remark}

\begin{lemma} \label{lemma:cauchyproduct}
	Let $(\varepsilon_{k})_{k \in \mathbb{N}}$ be in $\mathbb{R}_{+}$ and let $(\rho_{k})_{k \in \mathbb{N}}$ be in $\left[0,1\right]$. Define 
	\begin{align*}
	(\forall k \in \mathbb{N}) \quad \chi_{k}:=\prod^{k}_{i=0} \rho_{i} \text{ and } \xi_{k} := \sum^{k}_{i=0} \lr{ \prod^{k}_{j=i+1} \rho_{j} } \varepsilon_{i}.
	\end{align*}
	Suppose that $\sum_{k \in \mathbb{N}} \varepsilon_{k} <\infty$ and that $ \limsup_{k \to \infty} \lr{\chi_{k} }^{\frac{1}{k}} <1$.  Then the following hold. 
	\begin{enumerate}

		\item \label{lemma:cauchyproduct:chi} $\sum_{k \in \mathbb{N}} \chi_{k} < \infty$ and $\lim_{k \to \infty} \chi_{k} =0$.
		\item \label{lemma:cauchyproduct:sum} Suppose  that  $(\forall k \in \mathbb{N})$ $\rho_{k+1} \leq \rho_{k}$. Then $\sum_{k \in \mathbb{N}} \xi_{k} < \infty$. 
	\end{enumerate}
\end{lemma}

\begin{proof}

	\cref{lemma:cauchyproduct:chi}: Because $\chi := \limsup_{k \to \infty} \lr{\chi_{k} }^{\frac{1}{k}} <1$, there exists $\varepsilon >0$ and $K \in \mathbb{N}$ such that $\chi + \varepsilon <1$ and 
	\begin{align}\label{eq:lemma:cauchyproduct:chi:EQ}
	\lr{\forall k \geq K} \quad 
	\lr{\chi_{k} }^{\frac{1}{k}} < \chi + \varepsilon, \quad \text{that is,} \quad
	\chi_{k} < \lr{ \chi + \varepsilon}^{k}.
	\end{align}
	Hence,
	\begin{align*}
	(\forall k \geq K) \quad \sum^{k}_{i=0} \chi_{i} =  \sum^{K }_{i=0} \chi_{i} +\sum^{k}_{i=K +1} \chi_{i} 
	\stackrel{\cref{eq:lemma:cauchyproduct:chi:EQ}}{\leq }   
	\sum^{K }_{i=0} \chi_{i} +\sum_{i \in \mathbb{N}}  \lr{ \chi + \varepsilon}^{i} 
	= \sum^{K}_{i=0} \chi_{i} + \frac{1}{1- \lr{ \chi + \varepsilon}} < \infty,\footnotemark
	\end{align*}
	which yields that $\sum_{k \in \mathbb{N}} \chi_{k} < \infty$ and $\lim_{k \to \infty} \chi_{k} =0$.
	
		\footnotetext{In the whole work, we use the empty sum convention. Hence, if $k = K$ in this inequality, then  $\sum^{k}_{i=K +1} \chi_{i} =0$.}
		
	\cref{lemma:cauchyproduct:sum}: Suppose that there exists $\bar{k} \in \mathbb{N}$ such that $\rho_{ \bar{k} }=0$. Then, by the assumption $(\forall k \in \mathbb{N})$ $\rho_{k+1} \leq \rho_{k}$ and $\rho_{k} \geq 0$, we know that  $(\forall k \geq \bar{k} )$ $\rho_{k} =0$. So  $(\forall k \geq \bar{k})$  $\lr{ \forall i \in \{0, 1, \ldots, k  \}}$ $ \prod^{k}_{j=i+1} \rho_{j} =0$, which implies that $(\forall k \geq \bar{k})$ $  \xi_{k} =0$. Hence, in this case, it is trivial that $\sum_{k \in \mathbb{N}} \xi_{k} < \infty$.

	Suppose that $(\forall k \in \mathbb{N})$ $\rho_{k} > 0$.
	Notice that 
	\begin{align*}
	(\forall k \in \mathbb{N}) \quad  \xi_{k} = \sum^{k}_{i=0} \lr{\prod^{k}_{j=i+1} \rho_{j}} \varepsilon_{i}=  \sum^{k}_{i=0} \lr{  \lr{\prod^{k-i}_{j=0} \rho_{j}}  \varepsilon_{i} \cdot \frac{ \lr{\prod^{k}_{j=i+1} \rho_{j} }}{ \lr{\prod^{k-i}_{j=0} \rho_{j} }} } 
	\end{align*}
	and that for every $k \in \mathbb{N}$ and $i \in \{0,1,\ldots,k\}$
	\begin{align*}
	\frac{ \prod^{k}_{j=i+1} \rho_{j} }{\prod^{k-i}_{j=0} \rho_{j} } =&
	\begin{cases}
	\frac{ \prod^{k}_{j=i+1} \rho_{j} }{\prod^{k-i}_{j=0} \rho_{j} } \quad  &\text{if } k-i \leq i+1;\\
	\frac{ \prod^{k}_{j=k-i+1} \rho_{j} }{\prod^{i}_{j=0} \rho_{j} } \quad &\text{if } k-i > i+1,\\
	\end{cases}\\
	=&\begin{cases}
	\frac{1}{\rho_{0}} \prod^{k-i}_{j=1} \frac{  \rho_{j+i} }{ \rho_{j} } \quad &\text{if } k-i \leq i+1;\\
	\frac{1}{\rho_{0}}  \prod^{i}_{j= 1}  \frac{ \rho_{j+k-i} }{  \rho_{j} } \quad &\text{if } k-i > i+1,\\
	\end{cases}\\
	\leq & \frac{1}{\rho_{0}}, 
	\end{align*}
	where we invoke the assumption $(\forall k \in \mathbb{N})$ $\rho_{k+1} \leq \rho_{k}$  in the last inequality. 
	
	On the other hand,  due to \cite[Page~80]{Rudin1976}, the Cauchy product  $	\sum_{k \in \mathbb{N}}	\sum^{k}_{i=0}  \chi_{k-i} \varepsilon_{i} $ of the two absolutely convergent series $\sum_{k \in \mathbb{N}} \varepsilon_{k}  $ and $\sum_{k \in \mathbb{N}} \chi_{k}  $   is absolutely  convergent, that is, 
	\begin{align} \label{eq:lemma:cauchyproduct:sum:cauchy}
	\sum_{k \in \mathbb{N}}	\sum^{k}_{i=0}  \chi_{k-i} \varepsilon_{i} < \infty.
	\end{align}
	Taking all results obtained above into account, we derive that
	\begin{align*}
	\sum_{k \in \mathbb{N}} \xi_{k}  \leq \frac{1}{\rho_{0}} \sum_{k \in \mathbb{N}}  \sum^{k}_{i=0} \lr{ \prod^{k-i}_{j=0}  \rho_{j} } \varepsilon_{i} =  \frac{1}{\rho_{0}} \sum_{k \in \mathbb{N}}   \sum^{k}_{i=0}  \chi_{k-i} \varepsilon_{i} \stackrel{\cref{eq:lemma:cauchyproduct:sum:cauchy}}{<} \infty,
	\end{align*}
	which reaches the required inequality in \cref{lemma:cauchyproduct:sum}. 
\end{proof}

\section{Metrical Subregularity} \label{section:MetricalSubregularity}
Metrical subregularity is a critical assumption for the linear convergence  of algorithms studied in this work. In this section, we exhibit results related to metrical subregularity.
\begin{definition}  \label{definition:metricsubregularity}  {\rm \cite[Pages~183 and 184]{DontchevRockafellar2014}}
	Let $F: \mathcal{H} \to 2^{\mathcal{H}}$ be a set-valued operator. $F$ is called  \emph{metrically subregular at $\bar{x}$ for $\bar{y}$} if $\lr{ \bar{x}, \bar{y} } \in \gra F$ and there exist  $\kappa \in \mathbb{R}_{+}$ and a neighborhood $U$ of $\bar{x}$ such that  
	\begin{align*}
	(\forall x \in U) \quad	\dist \lr{ x, F^{-1} (\bar{y})} \leq \kappa \dist \lr{\bar{y}, F(x)}.
	\end{align*}
	The constant $\kappa$ is called \emph{constant of metric subregularity}.
	The infimum of all $\kappa$ for which the inequality above holds is the \emph{modulus of metric subregularity}, denoted by $\subreg \lr{F;\bar{x}|\bar{y}}$. The absence of metric subregularity is signaled by 
	$\subreg \lr{F;\bar{x}|\bar{y}} =\infty$.
\end{definition}

\begin{fact} \label{lemma:metricallysubregularEQ} {\rm \cite[Lemma~2.19]{OuyangStabilityKMIterations2022}}
	Let  $A: \mathcal{H} \to 2^{\mathcal{H}}$ be   maximally monotone with $\zer A \neq \varnothing$, let $\bar{x} \in \zer A$, and let $\gamma \in \mathbb{R}_{++}$.    Then $A$ is metrically subregular at $\bar{x}$ for $0 \in A\bar{x}$ if and only if $  \Id -\J_{\gamma A}  $ is metrically subregular at $\bar{x}$ for $0 = \lr{\Id -\J_{\gamma A}} \bar{x}$. In particular, if $A$ is metrically subregular at $\bar{x}$ for $0 \in A\bar{x}$, i.e., 
	\begin{align*} 
	(\exists \kappa >0) (\exists \delta >0) (\forall x \in B[\bar{x}; \delta]) \quad \dist \lr{x, A^{-1}0} \leq \kappa \dist \lr{0, Ax},
	\end{align*}
	then  $  \Id -\J_{\gamma A}  $ is metrically subregular at $\bar{x}$ for $0 = \lr{\Id -\J_{\gamma A} } \bar{x}$; more specifically,
	\begin{align*}
	(\forall x \in B[\bar{x}; \delta]) \quad 	\dist \lr{x, \lr{\Id -\J_{\gamma A}}^{-1}0} \leq \lr{1 + \frac{\kappa}{\gamma}} \dist \lr{0,  \lr{\Id -\J_{\gamma A}} x}.
	\end{align*}	
\end{fact}

	Let $A: \mathcal{H} \to 2^{\mathcal{H}}$ be maximally monotone with $\bar{x} \in \zer A$, let $ \gamma $ and $\delta$ be in $ \mathbb{R}_{++}$, and let $x \in \mathcal{H}$. If $x \in  B[\bar{x};\delta]$, then due to \cref{lemma:JGammaAFix},   $ \J_{\gamma A}x \in B[\bar{x};\delta]$. Hence, applying \cref{lemma:JgammaAINEQ}, we easily deduce   \cite[Theorem~3.1]{Leventhal2009} and \cite[Lemma~5.3]{TaoYuan2018}. 
In fact, the proofs of \cref{lemma:JgammaAINEQ},  \cite[Theorem~3.1]{Leventhal2009},   and \cite[Lemma~5.3]{TaoYuan2018} are similar. 
\begin{lemma} \label{lemma:JgammaAINEQ} 
	Let  $A: \mathcal{H} \to 2^{\mathcal{H}}$ be  maximally monotone with $\zer A \neq \varnothing$   and let $\gamma \in \mathbb{R}_{++}$.    Assume that $A$ is metrically subregular at $\bar{x}$ for $0 \in A\bar{x}$, i.e., 
	\begin{align} \label{eq:lemma:JgammaAINEQ:MSubr}
	(\exists \kappa >0) (\exists \delta >0)	(\forall x \in B[\bar{x};\delta]) \quad	\dist \lr{ x, A^{-1} 0} \leq \kappa \dist \lr{0, Ax}.
	\end{align}
	Then for every $x \in \mathcal{H}$ with $ \J_{\gamma A}x \in B[\bar{x};\delta]$,
	\begin{align} \label{eq:lemma:JgammaAINEQ:MetriSubR}  
	\dist \lr{ \J_{\gamma A}x, A^{-1} 0} \leq \frac{1}{ \sqrt{1 + \frac{\gamma^{2}}{\kappa^{2}} } } \dist \lr{x, A^{-1}0}.
	\end{align}	
\end{lemma}

\begin{proof}
	As a consequence of  \cref{corollary:JcA},   
	\begin{align}   \label{eq:lemma:JgammaAINEQ:graA} 
(\forall x \in \mathcal{H}) \quad \frac{1}{\gamma} \lr{ x- \J_{\gamma A} x } \in A \lr{ \J_{\gamma A} x }. 
	\end{align}
	Due to 	\cref{lemma:JGammaAFix},
	\begin{align} \label{eq:lemma:JgammaAINEQ:ZerA} 
	(\forall x \in \mathcal{H})	(\forall z \in \zer A) \quad \norm{\J_{\gamma A} x -z}^{2} +\norm{ \lr{\Id - \J_{\gamma A}}x}^{2} \leq \norm{x -z}^{2}.
	\end{align}
By virtue of the maximal monotonicity of $A$ and  via  \cite[Proposition~23.39]{BC2017}, we know that $\zer A $ is closed and convex. So $\Pro_{\zer A}x \in \zer A$ is well-defined. 
	Substitute $z=\Pro_{\zer A}x$ in \cref{eq:lemma:JgammaAINEQ:ZerA}  to establish that 
	\begin{align} \label{eq:lemma:JgammaAINEQ:ZerAP} 
	(\forall x \in \mathcal{H}) \quad  \norm{  x - \J_{\gamma A}x}^{2} \leq \norm{x -\Pro_{\zer A}x}^{2} -\norm{\J_{\gamma A} x - \Pro_{\zer A}x }^{2}.
	\end{align}
	
	Let $x \in \mathcal{H}$ with $ \J_{\gamma A}x \in B[\bar{x};\delta]$. Applying \cref{eq:lemma:JgammaAINEQ:MSubr} with $x =\J_{\gamma A}x $ in the first inequality below and employing $\norm{\J_{\gamma A} x - \Pro_{\zer A} \lr{ \J_{\gamma A} x } } \leq \norm{\J_{\gamma A} x - \Pro_{\zer A}x }$ in the fourth inequality, we deduce that 
	\begin{align*}
	\dist^{2} \lr{ \J_{\gamma A}x  , A^{-1} 0} ~\leq~ & \kappa^{2} \dist^{2} \lr{0, A \lr{\J_{\gamma A}x} } \\
	\stackrel{\cref{eq:lemma:JgammaAINEQ:graA} }{\leq }	& \frac{ \kappa^{2} }{ \gamma^{2}} \norm{x- \J_{\gamma A} x}^{2}\\
	\stackrel{\cref{eq:lemma:JgammaAINEQ:ZerAP}}{\leq}  & \frac{ \kappa^{2} }{ \gamma^{2}} \lr{ \norm{x -\Pro_{\zer A}x}^{2} -\norm{\J_{\gamma A} x - \Pro_{\zer A}x }^{2}}\\
	~\leq~& \frac{ \kappa^{2} }{ \gamma^{2}} \norm{x -\Pro_{\zer A}x}^{2} -  \frac{ \kappa^{2} }{ \gamma^{2}} \norm{\J_{\gamma A} x - \Pro_{\zer A} \lr{ \J_{\gamma A} x}  }^{2}\\
	~=~& \frac{ \kappa^{2} }{ \gamma^{2}}	\dist^{2} \lr{ x  , A^{-1} 0}  -  \frac{ \kappa^{2} }{ \gamma^{2}}	\dist^{2} \lr{ \J_{\gamma A}x  , A^{-1} 0}, 
	\end{align*}
	which yields \cref{eq:lemma:JgammaAINEQ:MetriSubR}  easily, since it is clear that $\frac{  \frac{ \kappa^{2} }{ \gamma^{2}} }{ 1 +  \frac{ \kappa^{2} }{ \gamma^{2}}}= \frac{1}{1 +  \frac{ \gamma^{2} }{ \kappa^{2}}}$.
\end{proof}

\begin{theorem}  \label{theorem:MetrSubregUpperBound}
	Let  $A: \mathcal{H} \to 2^{\mathcal{H}}$ be  maximally monotone with $\zer A \neq \varnothing$. Let $x \in \mathcal{H}$, let $e \in \mathcal{H}$ and $\eta \in \mathbb{R}_{+}$, 
	let $\lr{  \gamma  , \varepsilon} \in \mathbb{R}_{++}^{2}$, and 	let $\lambda \in \left]0,2\right[\,$. Define
		\begin{align*}
 	y_{x} := \lr{1-\lambda} x +\lambda \J_{\gamma A}x \quad \text{and} \quad z_{x} := \lr{1-\lambda} x +\lambda \J_{\gamma A}x +\eta e=y_{x} + \eta e. 
	\end{align*}  
	Suppose that $A$ is metrically subregular at $\bar{x}$ for $0 \in A\bar{x}$, i.e., 
	\begin{align} \label{eq:theorem:MetrSubregUpperBound:MetricSub} 
	(\exists \kappa >0) (\exists \delta >0) (\forall x \in B[\bar{x}; \delta]) \quad \dist \lr{x, A^{-1}0} \leq \kappa \dist \lr{0, Ax}.
	\end{align}
	Set $\rho  := \lr{ 1 - \lambda \lr{2- \lambda} \frac{1}{ \lr{1+ \frac{\kappa}{\gamma}}^{2} }}^{\frac{1}{2}}$.
Suppose that $\J_{\gamma A}x \in B[\bar{x}; \delta]$. Then the following statements hold. 
	\begin{enumerate}
		\item  \label{theorem:MetrSubregUpperBound:rho} $\rho  \in \left]0,1\right[\,$.
		\item \label{theorem:MetrSubregUpperBound:yx}  $\norm{y_{x} - \Pro_{\zer A}x}  \leq \rho  \norm{x - \Pro_{\zer A}x}$.
		
		\item \label{theorem:MetrSubregUpperBound:zx} Suppose that $\norm{e} \leq \varepsilon \norm{x-z_{x}}$ and $\eta \varepsilon \in \left[0,1\right[\,$.	Then
		\begin{align*}
		\norm{z_{x} - \Pro_{\zer A}x}  \leq \frac{ \rho  +\eta \varepsilon }{1- \eta \varepsilon} \norm{x - \Pro_{\zer A}x}.
		\end{align*}
	\end{enumerate}
\end{theorem}

\begin{proof}
	\cref{theorem:MetrSubregUpperBound:rho}: This is clear from $\lambda \in \left]0,2\right[\,$, $\lambda \lr{2- \lambda} = -(1-\lambda)^{2} +1 \leq 1$, and $\frac{1}{ \lr{1+ \frac{\kappa}{\gamma}}^{2}} \in \left]0,1\right[\,$.
	
	\cref{theorem:MetrSubregUpperBound:yx}: Applying \cref{eq:theorem:MetrSubregUpperBound:MetricSub} with $x=\J_{\gamma A}x$ in the first  inequality and employing  \cref{corollary:JcA}  in the   second inequality,  we know that 
	\begin{align} \label{theorem:MetrSubregUpperBound:yx:dist}
\norm{\J_{\gamma A}x -  \Pro_{\zer A} \lr{\J_{\gamma A}x}} =	\dist \lr{ \J_{\gamma A}x, A^{-1}0 } \leq \kappa \dist \lr{0, A \lr{\J_{\gamma A}x}} \leq \frac{\kappa}{\gamma} \norm{x -  \J_{\gamma A}x}.
	\end{align}
	Hence,
	\begin{subequations} \label{theorem:MetrSubregUpperBound:yx:kappagamma}
		\begin{align} 
		\norm{x - \Pro_{\zer A}x} ~\leq~&  \norm{x - \Pro_{\zer A} \lr{\J_{\gamma A}x}} \\
		~\leq~&  \norm{x - \J_{\gamma A}x} +\norm{\J_{\gamma A}x -  \Pro_{\zer A} \lr{\J_{\gamma A}x}} \\
		\stackrel{\cref{theorem:MetrSubregUpperBound:yx:dist}}{\leq}& \lr{1+\frac{\kappa}{\gamma}}  \norm{x - \J_{\gamma A}x}. 
		\end{align}
	\end{subequations}
	Applying \cref{lemma:resolvents:yxzx} with $\bar{x} = \Pro_{\zer A}x$ in the first inequality below, we observe that
	\begin{align*}
	\norm{y_{x} - \Pro_{\zer A}x }^{2} ~\leq~&  \norm{x -  \Pro_{\zer A}x }^{2} -\lambda \lr{2 -\lambda} \norm{x -\J_{\gamma A}x}^{2}\\
	\stackrel{\cref{theorem:MetrSubregUpperBound:yx:kappagamma}}{\leq} &  \lr{ 1 - \lambda \lr{2- \lambda} \frac{1}{ \lr{1+ \frac{\kappa}{\gamma}}^{2} }} \norm{x -  \Pro_{\zer A}x }^{2}, 
	\end{align*}
	which guarantees the desired inequality in \cref{theorem:MetrSubregUpperBound:yx}.
	
	\cref{theorem:MetrSubregUpperBound:zx}: 
	Based on \cref{cor:fact:FixJcAzerA}, $\J_{\gamma A}$ is $\frac{1}{2}$-averaged and $	\Fix \J_{\gamma A} = \zer A \neq \varnothing$. Employing  \cref{theorem:MetrSubregUpperBound:yx}  and applying \cref{lemma:yxzx} with $T= \J_{\gamma A}$, $\alpha=\frac{1}{2}$, $\beta =\rho$, and $\bar{x} = \Pro_{\zer A}x$, we deduce \cref{theorem:MetrSubregUpperBound:zx}.
\end{proof}

\begin{remark} \label{remark:MetrSubregUpperBound}
	We uphold the assumptions of \cref{theorem:MetrSubregUpperBound}. 	By some easy algebra, it is not difficult to get that 
	\begin{align*}
	 	\max \left\{  \lr{1-\frac{\lambda}{ \frac{\kappa}{\gamma}+1}}^{2}, \lr{ 1 - \lambda \lr{2-\lambda} \frac{1}{1+\frac{\kappa^{2} }{\gamma^{2} }} }  \right\}  
	<   1 - \lambda \lr{2- \lambda} \frac{1}{ \lr{1+ \frac{\kappa}{\gamma}}^{2} }.
	\end{align*}
	Note that  if $\zer A $ is a singleton, then $\Pro_{\zer A}x = \Pro_{\zer A} \lr{ \J_{\gamma A}x }$.
	Therefore, based on \cref{theorem:MetrSubregOptimalBoundszx}\cref{theorem:MetrSubregOptimalBoundszx:MetricSub} below, if $\zer A $ is a singleton, then the coefficient $\rho$ in  \cref{theorem:MetrSubregUpperBound}\cref{theorem:MetrSubregUpperBound:yx}  can be decreased.
\end{remark}	

\begin{definition} \label{definition:A-1Lipschitze}   {\rm \cite[Page~885]{Rockafellar1976}}
	Let $F: \mathcal{H} \to 2^{\mathcal{H}}$ be a set-valued operator. We say $F^{-1}$ is \emph{Lipschitz continuous at $0$ $\lr{\text{with modulus } \alpha \geq 0}$ } if there is a unique solution $\bar{z}$ to $0 \in F(z)$ $\lr{\text{i.e.\,}F^{-1}(0)=\{\bar{z}\}}$, and for some $\tau >0$ we have 
	\begin{align*}
	\norm{z -\bar{z}} \leq \alpha \norm{w} \quad \text{whenever } z \in F^{-1}(w) \text{ and } \norm{w} \leq \tau. 
	\end{align*}
	\index{Lipschitz continuous at $0$} \index{modulus}
\end{definition}

Let $A : \mathcal{H} \to 2^{\mathcal{H}}$ be a maximally monotone operator with $\zer A \neq \varnothing$.  It was claimed in 
\cite[Page~684]{Leventhal2009} and \cite[Page~5]{ShenPan2016} without proof   that   the assumption that $A^{-1}$  is Lipschitz continuous at $0$ is stronger than that $A$ is metrically subregular. For completeness, we provide a detailed proof for this claim below. 
\begin{fact} \label{prop:LipschitzMetricSubregularity}
	Let $A :\mathcal{H} \to 2^{\mathcal{H}}$  be maximally monotone with $\zer A \neq \varnothing$. Suppose that $A^{-1}$ is Lipschitz continuous at $0$ with modulus $\alpha > 0$, i.e., $A^{-1}(0) =\{\bar{x}\}$ and there exists $\tau >0$   such that 
	\begin{align} \label{eq:prop:LipschitzMetricSubregularity}
	\lr{ \forall (w,x) \in \gra A^{-1} \text{ with } w \in B[0;\tau]} \quad \norm{x -\bar{x}} \leq \alpha \norm{w}.
	\end{align}
	Set $\delta := \alpha \tau$. 	Then $A$ is metrically subregular at $\bar{x}$ for $0 \in A \bar{x}$ with $\subreg \lr{A;\bar{x}|\bar{y}} \leq \alpha$; more precisely,
	\begin{align} \label{eq:prop:LipschitzMetricSubregularity:MSub}
	\lr{\forall x \in B[\bar{x}; \delta]} \quad \dist \lr{x,  A^{-1}  0} \leq \alpha \dist \lr{0, Ax}.
	\end{align}
\end{fact}

\begin{proof}
	Let $x \in B[\bar{x}; \delta]$. Because $A^{-1}  0 =\{\bar{x}\}$, we know that 
	\begin{align} \label{eq:prop:LipschitzMetricSubregularity:norm}
	\dist \lr{x,  A^{-1}  0} =\norm{x -\bar{x}}.
	\end{align}
	
	If $Ax=\varnothing$, then, via the convention $\inf \varnothing = \infty$, $	\dist \lr{x,  A^{-1}  0}  \leq \infty = \alpha  \dist \lr{0, Ax}$. 
	
	Assume that $Ax \neq \varnothing$.  Then, invoking  the maximal monotonicity of $A$ and employing   \cite[Proposition~23.39]{BC2017}, we know that $Ax$ is nonempty closed and convex. So, via \cite[Theorem~3.16]{BC2017},  $y:=\Pro_{Ax} 0$ is a well-defined point in $Ax$. Then we have exactly the following two cases.
	
	\emph{Case~1:} $y \in B[0;\tau]$. Now $(y,x) \in \gra A^{-1}$ with $y \in B[0;\tau]$.  Bearing  \cref{eq:prop:LipschitzMetricSubregularity} and \cref{eq:prop:LipschitzMetricSubregularity:norm} in mind, we observe that 
	\begin{align*}
	\dist \lr{x,  A^{-1}  0} =	\norm{x -\bar{x}} \leq \alpha \norm{y} = \alpha \norm{\Pro_{Ax} 0} = \alpha \dist \lr{0, Ax}.
	\end{align*} 
	
	\emph{Case~2:} $y \notin B[0;\tau] $. Then $\norm{y}=\norm{\Pro_{Ax} 0}>\tau$. Hence, it is easy to see that
	\begin{align*}
	\dist \lr{x,  A^{-1}  0} =	\norm{x -\bar{x}} \leq \delta =\alpha \tau \leq \alpha  \norm{\Pro_{Ax}  0} = \alpha \dist \lr{0, Ax}.
	\end{align*}
	
	Altogether, \cref{eq:prop:LipschitzMetricSubregularity:MSub} is true in both cases and the proof is complete. 
\end{proof}

 Naturally, one may have the following question. 
 \begin{question}\label{question:MetriSub}
 	Let $A :\mathcal{H} \to 2^{\mathcal{H}}$  be maximally monotone. Suppose that  $\zer A =\{\bar{x}\}$. Are the following two statements equivalent?
 	\begin{enumerate}
 		\item \label{question:MetriSub:MetSub} $A$ is metrically subregular at $\bar{x}$ for $0 \in A \bar{x}$.
 		\item \label{question:MetriSub:Lip} $A^{-1}$ is Lipschitz continuous at $0$ with a positive modulus.
 	\end{enumerate} 
 	Based on \cref{prop:LipschitzMetricSubregularity}, we know that \cref{question:MetriSub:Lip} $\Rightarrow$ \cref{question:MetriSub:MetSub}.
 	Therefore, \cref{question:MetriSub:MetSub} $\Leftrightarrow$ \cref{question:MetriSub:Lip}  if and only if  \cref{question:MetriSub:MetSub} $\Rightarrow$ \cref{question:MetriSub:Lip}.
 \end{question}
 
 \cref{example:fmetricsub} illustrates that for all continuous and monotone 
 function $f : \mathbb{R} \to \mathbb{R}$ with $f^{-1}(0)=\{\bar{x}\}$,    the metrical subregularity of $f$ is equivalent to the Lipschitz continuity of $f^{-1}$ at $0$ with a positive modulus. In other words, we provide a specific example  showing the equivalence of the two statements in  \cref{question:MetriSub}.
 \begin{example} \label{example:fmetricsub}
 	Let $f : \mathbb{R} \to \mathbb{R}$ be continuous and monotone and let $f^{-1}(0)=\{\bar{x}\}$. Then the following hold. 
 	\begin{enumerate}
 		\item \label{example:fmetricsub:maximallymonotone} $f$ is maximally monotone. 
 		\item \label{example:fmetricsub:imply} Let $\epsilon \in \mathbb{R}_{++}$. Then there exists $\delta \in \mathbb{R}_{++}$ such that 
 		\begin{align} \label{eq:example:fmetricsub:imply}
 		f(x) \in B[0;\delta]  \Rightarrow 	x \in B[\bar{x}; \epsilon].
 		\end{align}
 		\item \label{example:fmetricsub:Lipschitz} $f$ is metrically subregular at $\bar{x}$ for $0 =  f (\bar{x})$ if and only if $f^{-1}$ is Lipschitz continuous at $0$ with a positive modulus.
 	\end{enumerate}
 \end{example}
 
 \begin{proof}
 	\cref{example:fmetricsub:maximallymonotone}: Because $f$ is  continuous and monotone, via \cite[Corollary~20.28]{BC2017}, $f$ is maximally monotone.

 	\cref{example:fmetricsub:imply}: Suppose to the contrary that \cref{eq:example:fmetricsub:imply} is not true. Then 
 	\begin{align*} 
 	\lr{\forall k \in \mathbb{N} \smallsetminus \{0\}} \text{ there exists } x_{k} \in \mathbb{R} \text{ such that }
  f(x_{k}) \in B\left[0;\frac{1}{k}\right] \text{ and } x_{k} \notin B\left[\bar{x}; \epsilon\right],
 	\end{align*}
 	which forces that
 	\begin{align} \label{eq:example:fmetricsub:imply:fxk}
 	f(x_{k}) \to 0 \quad \text{and} \quad \Omega \lr{ \lr{x_{k}}_{k \in \mathbb{N}} } \cap  B\left[\bar{x}; \frac{\epsilon}{2}\right] =\varnothing,
 	\end{align}
 	where $\Omega \lr{ \lr{x_{k}}_{k \in \mathbb{N}} }$ is the set of all sequential cluster points of $\lr{x_{k}}_{k \in \mathbb{N}}$.
 	
 	We have exactly the following cases. 
 	
 	\emph{Case~1}: $(x_{k})_{k \in \mathbb{N}}$ is bounded. Then this together with \cref{eq:example:fmetricsub:imply:fxk} implies that there exists a subsequence $(x_{k_{i}})_{i \in \mathbb{N}}$ of $(x_{k})_{k \in \mathbb{N}}$ such that 
 	\begin{align*}
 	f(x_{k_{i}}) \to 0 \quad \text{and} \quad x_{k_{i}} \to \hat{x} \notin  B\left[\bar{x}; \frac{\epsilon}{2}\right]. 
 	\end{align*}
 	On the other hand, the continuity of $f$ implies that 
 	\begin{align*}
 	f(\hat{x}) = f\lr{\lim_{i \to \infty} x_{k_{i}} }  = \lim_{i \to \infty} f(x_{k_{i}})  =0,
 	\end{align*}
 	which contradicts  that $f^{-1}(0)=\{\bar{x}\}$ and $\hat{x} \neq \bar{x}$.
 	
 	\emph{Case~2}: $(x_{k})_{k \in \mathbb{N}}$ is not bounded. Without loss of generality, we assume that there exists a subsequence $(x_{k_{i}})_{i \in \mathbb{N}}$ of $(x_{k})_{k \in \mathbb{N}}$ such that 
 	\begin{align} \label{eq:example:fmetricsub:imply:xki}
 	x_{k_{i}} \to \infty.
 	\end{align}
 	Let $\tilde{x} > \bar{x}$. 	 (If $x_{k_{i}} \to -\infty$, then we choose $\tilde{x} < \bar{x}$ and the remaining proof is similar to  the following proof.)  Inasmuch as  $f^{-1}(0)=\{\bar{x}\}$, we have exactly the following two subcases.
 	
 	\emph{Subcase~2.1}: $f(\tilde{x}) >0$.   Combine this with \cref{eq:example:fmetricsub:imply:fxk} and  \cref{eq:example:fmetricsub:imply:xki} to deduce that  there exists $N \in \mathbb{N}$ such that 
 	\begin{align*}
 	\lr{\forall i \geq N} \quad 	 f(x_{k_{i}}) <  f(\tilde{x}) \quad \text{and} \quad  x_{k_{i}} > \tilde{x},
 	\end{align*}
 	which entails that
 	\begin{align*}
 	\innp{ \tilde{x}-\bar{x},  f(\tilde{x}) -f(\bar{x})} >0 \quad \text{and} \quad \innp{ x_{k_{N}} -\tilde{x}, f\lr{x_{k_{N}}} - f\lr{\tilde{x}}} <0.
 	\end{align*}
 	This contradicts the monotonicity of $f$.

 	\emph{Subcase~2.2}: $f(\tilde{x}) <0$. Similarly, as a consequence of \cref{eq:example:fmetricsub:imply:fxk}  and \cref{eq:example:fmetricsub:imply:xki},  
 	there exists $N \in \mathbb{N}$ such that 
 	\begin{align*}
 	\lr{\forall i \geq N} \quad 	 f(x_{k_{i}}) >  f(\tilde{x}) \quad \text{and} \quad  x_{k_{i}} > \tilde{x}.
 	\end{align*}
 	This necessitates that
 	\begin{align*}
 	\innp{ \tilde{x}-\bar{x},  f(\tilde{x}) -f(\bar{x})} < 0 \quad \text{and} \quad \innp{ x_{k_{N}} -\tilde{x}, f\lr{x_{k_{N}}} - f\lr{\tilde{x}}} >0,
 	\end{align*}
 	which contradicts the monotonicity of $f$ as well.
 	
 	Altogether, \cref{example:fmetricsub:imply} holds in all cases.
 	
 	\cref{example:fmetricsub:Lipschitz}: Suppose that $f$ is metrically subregular at $\bar{x}$ for $0 =  f (\bar{x})$. Then
 	\begin{align}  \label{eq:example:fmetricsub:Lipschitz}
(\exists \kappa >0) (\exists \delta >0)	(\forall x \in B[\bar{x};\delta]) \quad 	\dist \lr{ x, f^{-1} (0) } \leq \kappa \dist \lr{0, f(x)},  \text{ that is, }
   \abs{x -\bar{x}} \leq \kappa \abs{f(x)}.
 	\end{align}
 	As a result of  \cref{prop:LipschitzMetricSubregularity}, it remains to prove that $f^{-1}$ is Lipschitz continuous at $0$ with a positive modulus.
 	
 Replace $\epsilon =\delta$ in \cref{example:fmetricsub:imply} above to see that there exists $\delta' \in \mathbb{R}_{++}$ such that 
 	\begin{align} \label{eq:example:fmetricsub:Lipschitz:imply}
 	f(x) \in B[0;\delta']  \Rightarrow 	x \in B[\bar{x}; \delta].
 	\end{align}
 	Now, invoke \cref{eq:example:fmetricsub:Lipschitz} and \cref{eq:example:fmetricsub:Lipschitz:imply} to obtain that
 	\begin{align*}
 	\lr{ \forall (f(x),x) \in \gra f^{-1} \text{ with } f(x) \in B[0; \delta']} \quad \abs{x -\bar{x}} \leq \kappa \abs{f(x)},
 	\end{align*}
 	which, via \cref{definition:A-1Lipschitze}, ensures that  $f^{-1}$ is Lipschitz continuous at $0$ with  modulus $\kappa >0$.
 	
 	Altogether, the proof is complete. 
 \end{proof}

 \cref{theorem:MetrSubregOptimalBounds}  together with  \cref{theorem:MetrSubregUpperBound} will play a critical role to prove the linear convergence of generalized proximal point algorithms later. 
 \begin{theorem} \label{theorem:MetrSubregOptimalBounds}
 	Let  $A: \mathcal{H} \to 2^{\mathcal{H}}$ be   maximally monotone with $\zer A \neq \varnothing$, let $x \in \mathcal{H}$,  let $\gamma \in \mathbb{R}_{++}$, and 	let $\lambda \in \left[0,2\right]$. Define
 	\begin{align}\label{eq:theorem:MetrSubregOptimalBounds:yx}
 	y_{x} := \lr{1-\lambda} x +\lambda \J_{\gamma A}x.
 	\end{align}  
 	Let $z \in \zer A$ and let $t \in \mathbb{R}_{++}$.
 	Suppose  that
 	\begin{align} \label{eq:theorem:MetrSubregOptimalBounds:leq1} 
 	\norm{ \J_{\gamma A}x -z } \leq t \norm{x -  \J_{\gamma A}x}.
 	\end{align} 
 	Then the following statements  hold. 
 	\begin{enumerate}
 		
 		\item  \label{theorem:MetrSubregOptimalBounds:leq1}  Suppose that  $\lambda \in \left[0,1\right]$. Then 
 		\begin{subequations}
 			\begin{align*}
 		 \norm{y_{x} -z}^{2}  
 			\leq  
 			\begin{cases}
 			\lr{1-\frac{\lambda}{t+1}}^{2} \norm{x -z}^{2} \quad &\text{if } t^{2} +\lambda -1 \leq 0;\\
 			\lr{ 1 - \lambda \lr{2-\lambda} \frac{1}{1+t^{2}} } \norm{x-z}^{2} \quad &\text{if } t^{2} +\lambda -1 > 0.
 			\end{cases}
 			\end{align*}
 		\end{subequations}
 		\item \label{theorem:MetrSubregOptimalBounds:geq1}   Suppose that $\lambda \in \left[1,2\right]$.  Then
 		\begin{align*}
 		\norm{y_{x} -z}^{2} \leq \lr{ 1 - \lambda \lr{2-\lambda} \frac{1}{1+t^{2}} } \norm{x-z}^{2}.
 		\end{align*}

 		\item \label{theorem:MetrSubregOptimalBounds:general}  	
 		Define
 		\begin{align}  \label{eq:theorem:MetrSubregOptimalBounds:general:rho:a}
 		\rho:=  	\begin{cases}
 		\lr{1-\frac{\lambda}{t+1}}^{2}   \quad &\text{if } 0\leq  \lambda \leq 1-t^{2};\\
 	 1 - \lambda \lr{2-\lambda} \frac{1}{1+t^{2}}   \quad &\text{if } 1-t^{2} < \lambda \leq 2.
 		\end{cases} 
 		\end{align}
 		Suppose that  $\lambda \in \left] 0,2\right[\,$.
 		Then 
 		\begin{align}  \label{eq:theorem:MetrSubregOptimalBounds:general:rho:b}
 		\rho =   \max \left\{ \lr{1-\frac{\lambda}{t+1}}^{2}, 1 - \lambda \lr{2-\lambda} \frac{1}{1+t^{2}}    \right \} \in  \left]0,1\right[\,.
 		\end{align} 
 		Moreover,
 		\begin{align*}   
 		\norm{y_{x} -z}^{2} \leq \rho  \norm{x-z}^{2}.
 		\end{align*}		
 	\end{enumerate}
 \end{theorem}
 
 \begin{proof}
 	Invoke \cref{lemma:JGammaAFix} in the following second inequality to entail that 
 	\begin{align*}
 	\lr{1+ \frac{1}{t^{2}}} \norm{\J_{\gamma A} x -z}^{2}  \stackrel{\cref{eq:theorem:MetrSubregOptimalBounds:leq1}}{\leq}\norm{\J_{\gamma A} x -z}^{2} +\norm{ x- \J_{\gamma A}x}^{2} \leq \norm{x -z}^{2},
 	\end{align*}
 	which necessitates that 
 	\begin{align}\label{eq:theorem:MetrSubregOptimalBounds:leq1:t2}
 	\norm{\J_{\gamma A} x -z}^{2}  \leq \frac{1 }{  1+ \frac{1}{t^{2}} }  \norm{x -z}^{2}.
 	\end{align}

 	\cref{theorem:MetrSubregOptimalBounds:leq1}:  If $\lambda =0$, then $\norm{y_{x} -z}^{2}  =\norm{x -z}^{2} =\lr{1-\frac{\lambda}{t+1}}^{2} \norm{x -z}^{2}  = \lr{ 1 - \lambda \lr{2-\lambda} \frac{1}{1+t^{2}} } \norm{x-z}^{2}$. Hence, the result in \cref{theorem:MetrSubregOptimalBounds:leq1} is trivial. 
 	
Suppose that  $\lambda \in \left]0,1\right]$. Then
 	\begin{align} \label{eq:theorem:MetrSubregOptimalBounds:leq1:EQ}
 	\lr{\lambda  -\lambda \lr{1 - \lambda} \frac{1}{t^{2}}} \leq 0 \Leftrightarrow t^{2} +\lambda -1 \leq 0.
 	\end{align}
 	It is clear that
 	\begin{subequations}  \label{eq:theorem:MetrSubregOptimalBounds:leq1:lambda}
 		\begin{align}
 		   \lr{1-\lambda}   + \lr{\lambda  -\lambda \lr{1 - \lambda} \frac{1}{t^{2}}} \frac{1 }{  1+ \frac{1}{t^{2}} }  
 	&	=  1+\lambda \lr{\frac{1 }{  1+ \frac{1}{t^{2}} } -1 } -\lambda \lr{1-\lambda}  \frac{1}{t^{2}}  \frac{1 }{  1+ \frac{1}{t^{2}} } \\
 	&	= 1 - \lambda \frac{1}{1+t^{2}} -\lambda \lr{1-\lambda} \frac{1}{1+t^{2}}\\ 
 	&	=  1 - \lambda \lr{2-\lambda} \frac{1}{1+t^{2}}.
 		\end{align}
 	\end{subequations}

 	Employing  \cite[Corollary~2.15]{BC2017}   in the following second equality and invoking  both \cref{eq:theorem:MetrSubregOptimalBounds:leq1:t2} and \cref{eq:theorem:MetrSubregOptimalBounds:leq1:EQ} in the second inequality, we observe that 	
 	\begin{align*}
  \norm{y_{x} -z}^{2} \stackrel{\cref{eq:theorem:MetrSubregOptimalBounds:yx}}{=} & \norm{ \lr{1-\lambda} (x-z) +\lambda \lr{\J_{\gamma A}x -z}}^{2}\\ 
 	~=~& \lr{1-\lambda} \norm{x -z}^{2} +\lambda \norm{ \J_{\gamma A}x -z }^{2} -\lambda \lr{1 - \lambda} \norm{x - \J_{\gamma A}x}^{2}\\
 	\stackrel{\cref{eq:theorem:MetrSubregOptimalBounds:leq1}}{\leq} & \lr{1-\lambda} \norm{x -z}^{2} +\lambda \norm{ \J_{\gamma A}x -z }^{2} -\lambda \lr{1 - \lambda} \frac{1}{t^{2}}  \norm{ \J_{\gamma A}x -z}^{2}\\
 	~=~&   \lr{1-\lambda} \norm{x -z}^{2} + \lr{\lambda  -\lambda \lr{1 - \lambda} \frac{1}{t^{2}}} \norm{ \J_{\gamma A}x -z}^{2}\\
 	~\leq~&
 	\begin{cases}
 	\lr{1-\lambda} \norm{x -z}^{2}  \quad &\text{if } t^{2} +\lambda -1 \leq 0;\\
 	\lr{ \lr{1-\lambda}   + \lr{\lambda  -\lambda \lr{1 - \lambda} \frac{1}{t^{2}}} \frac{1 }{  1+ \frac{1}{t^{2}} } } \norm{x -z}^{2}  \quad &\text{if } t^{2} +\lambda -1 > 0.
 	\end{cases}
 	\end{align*}
 	This combined with \cref{eq:theorem:MetrSubregOptimalBounds:leq1:lambda} guarantees that 
 	\begin{align}\label{eq:theorem:MetrSubregOptimalBounds:leq1:norm}
 	\norm{y_{x} -z}^{2}  \leq \begin{cases}
 	\lr{1-\lambda} \norm{x -z}^{2}  \quad &\text{if } t^{2} +\lambda -1 \leq 0;\\
 	\lr{ 1 - \lambda \lr{2-\lambda} \frac{1}{1+t^{2}} } \norm{x -z}^{2}  \quad &\text{if } t^{2} +\lambda -1 > 0.
 	\end{cases}
 	\end{align}
 	On the other hand,  if $t^{2} +\lambda -1 \leq 0$, 	then utilizing \cref{eq:theorem:MetrSubregOptimalBounds:leq1}  and applying \cref{lemma:uvINEQ}\cref{lemma:uvINEQ:RestrictLambda}  with $u= x-z $ and $v= \J_{\gamma A}x -z$, we know that 
 	\begin{align} \label{eq:theorem:MetrSubregOptimalBounds:leq1:leq}
 	\norm{y_{x} -z}^{2}
 	\leq  \lr{ 1-\frac{\lambda}{t+1} }^{2} \norm{x-z}^{2}.
 	\end{align}

Notice that, by some easy algebra,  $ \lr{ 1-\frac{\lambda}{t+1} }^{2} \leq 1- \lambda \Leftrightarrow t^{2} +\lambda -1 \leq 0$. Hence, combining  \cref{eq:theorem:MetrSubregOptimalBounds:leq1:norm} and \cref{eq:theorem:MetrSubregOptimalBounds:leq1:leq}, we deduce \cref{theorem:MetrSubregOptimalBounds:leq1}. 
 	
 	\cref{theorem:MetrSubregOptimalBounds:geq1}: Because $\lambda \in \left[1,2\right]$, we have that 
 	\begin{align} \label{eq:theorem:MetrSubregOptimalBounds:geq1:lambda<}
 	\lambda \lr{1-\lambda} \leq 0.
 	\end{align}
 	Due to \cref{corollary:JcA},  $	\lr{ \J_{\gamma A} x , \frac{1}{\gamma} \lr{x-\J_{\gamma A} x}}  \in \gra A$.   This combined with $\lr{z,0} \in \gra A$  and the monotonicity of $A$ entails that 
 	\begin{align} \label{eq:theorem:MetrSubregOptimalBounds:geq1:monotone}
 	\innp{ \J_{\gamma A} x -z, x-\J_{\gamma A} x} \geq 0.
 	\end{align}
 	Invoke  \cref{eq:theorem:MetrSubregOptimalBounds:geq1:lambda<} and \cref{eq:theorem:MetrSubregOptimalBounds:geq1:monotone} in the following first inequality to derive that 
 	\begin{align*}
 	 & \norm{y_{x} -z}^{2}  \\
 	\stackrel{\cref{eq:theorem:MetrSubregOptimalBounds:yx}}{=}& \norm{\lr{1-\lambda} \lr{x-z} +\lambda \lr{\J_{\gamma A}x-z}}^{2}\\
 	~=~& \lr{1-\lambda}^{2} \norm{x-z}^{2} +  \lambda^{2} \norm{\J_{\gamma A}x-z}^{2} +2\lambda \lr{1-\lambda}  \innp{x-z,\J_{\gamma A}x-z}\\
 	~=~&\lr{1-\lambda}^{2} \norm{x-z}^{2} +  \lambda^{2} \norm{\J_{\gamma A}x-z}^{2} +2\lambda \lr{1-\lambda} \norm{\J_{\gamma A}x-z}^{2} 
 	 +2\lambda \lr{1-\lambda}  \innp{x-\J_{\gamma A}x,\J_{\gamma A}x-z}\\
 	~\leq~&\lr{1-\lambda}^{2} \norm{x-z}^{2} +  \lambda \lr{2-\lambda} \norm{\J_{\gamma A}x-z}^{2} \\
 	\stackrel{\cref{eq:theorem:MetrSubregOptimalBounds:leq1:t2}}{\leq}&
 	\lr{1-\lambda}^{2} \norm{x-z}^{2} +  \lambda \lr{2-\lambda} \frac{1 }{  1+ \frac{1}{t^{2}} }  \norm{x -z}^{2} \\
 	~=~& \lr{1 - \lambda \lr{2-\lambda} \frac{1}{1+t^{2}}}\norm{x -z}^{2}, 
 	\end{align*}
 	where the last equality follows from 
 	\begin{align*}
 	\lr{1-\lambda}^{2} +\lambda \lr{2-\lambda}  \frac{1 }{  1+ \frac{1}{t^{2}} } =
 	1+\lambda \lr{2-\lambda} \lr{-1 + \frac{1 }{  1+ \frac{1}{t^{2}} }  }
 	 = 1 - \lambda \lr{2-\lambda} \frac{1}{1+t^{2}}.
 	\end{align*}

 	\cref{theorem:MetrSubregOptimalBounds:general}: 
 	Inasmuch as $\lambda \in \left[0,2\right]$ and $t \in \mathbb{R}_{++}$, it is easy to get that
 	\begin{align*}
 	&\lr{1-\frac{\lambda}{t+1}}^{2} \in \left[0,1\right[  \Leftrightarrow  \lambda < 2\lr{t+1}; \\
 	& 1 - \lambda \lr{2-\lambda} \frac{1}{1+t^{2}}   \in \left]0,1\right[  \Leftrightarrow   \lambda \lr{2-\lambda}  >0.
 	\end{align*}
 	Hence, $\lambda \in \left]0,2\right[$ and $t \in \mathbb{R}_{++}$ lead to 
 	\begin{align*}
 	\max \left\{ \lr{1-\frac{\lambda}{t+1}}^{2}, 1 - \lambda \lr{2-\lambda} \frac{1}{1+t^{2}}    \right\} \in  \left]0,1\right[\,.
 	\end{align*}
 	Combine this with  \cref{fact:lambdat2} and \cref{eq:theorem:MetrSubregOptimalBounds:general:rho:a} to yield   \cref{eq:theorem:MetrSubregOptimalBounds:general:rho:b}.
 	
 	Furthermore, the last assertion in \cref{theorem:MetrSubregOptimalBounds:general}   is clear from \cref{theorem:MetrSubregOptimalBounds:leq1} and \cref{theorem:MetrSubregOptimalBounds:geq1} above.
 \end{proof}
 
 The inequality \cref{eq:remark:theorem:MetrSubregOptimalBounds} presented in \cref{remark:theorem:MetrSubregOptimalBounds} will be used to compare convergence rates of generalized proximal point algorithms later. 
 \begin{remark} \label{remark:theorem:MetrSubregOptimalBounds}
 	Let $\lambda \in \mathbb{R}_{+}$ and let $t \in \mathbb{R}_{++}$ such that 
   $\lambda \leq 1-t^{2}$. Then
 	\begin{align*}
 	 1 - \lambda \frac{1}{1+ t^{2}} - \lr{1-\frac{\lambda}{ t+1}}^{2}  
 &	= 	 \frac{2\lambda}{ t+1}   - \lambda \frac{1}{1+ t^{2}}  - \frac{\lambda^{2}}{( t+1)^{2}} \\
 &	=  \frac{\lambda}{(1+ t^{2}) ( t+1)^{2}} \lr{2( t+1) (1+ t^{2})  - ( t+1)^{2} - \lambda (1+ t^{2}) }\\
 &	\geq   \frac{\lambda}{(1+ t^{2}) ( t+1)^{2}} \lr{2( t+1) (1+ t^{2})  - ( t+1)^{2} - (1-t^{2}) (1+ t^{2}) }\\
 &	=  \frac{\lambda}{(1+ t^{2}) ( t+1)^{2}} \lr{2t^{3} +t^{2} +t^{4} } =\frac{\lambda t^{2}}{1+t^{2}} \geq  0,
 	\end{align*}
 	where we use the assumption $\lambda \leq 1-t^{2}$ in the first inequality above. 
 	
 	Hence, based on \cref{eq:theorem:MetrSubregOptimalBounds:general:rho:b} in \cref{theorem:MetrSubregOptimalBounds}\cref{theorem:MetrSubregOptimalBounds:general}, we know that 
 	\begin{align} \label{eq:remark:theorem:MetrSubregOptimalBounds}
 	 \max \left\{ \lr{1-\frac{\lambda}{t+1}}^{2},   1 - \lambda \lr{2-\lambda} \frac{1}{1+t^{2}}     \right\} 
 	\leq   \max \left\{ 1 - \lambda   \frac{1}{1+t^{2}} ,   1 - \lambda \lr{2-\lambda} \frac{1}{1+t^{2} }   \right\}.
 	\end{align}
 \end{remark}
 
 \begin{theorem}  \label{theorem:MetrSubregOptimalBoundszx}
 Let  $A: \mathcal{H} \to 2^{\mathcal{H}}$ be   maximally monotone with $\zer A \neq \varnothing$. Let $x $ and $e$ be in $ \mathcal{H}$, let $\bar{x} \in \zer A$, let  $\eta $ and $\varepsilon$ be in $ \mathbb{R}_{+}$,  let  $ \gamma \in \mathbb{R}_{++}$, and   let $\lambda \in \left]0,2\right[\,$. Define
 	\begin{align}
 	y_{x} := \lr{1-\lambda} x +\lambda \J_{\gamma A}x \quad \text{and} \quad z_{x} := \lr{1-\lambda} x +\lambda \J_{\gamma A}x +\eta e.
 		\end{align}
 Then the following statements hold. 
 \begin{enumerate}
 	\item \label{theorem:MetrSubregOptimalBoundszx:MetricSub}  Suppose that $A$ is metrically subregular at $\bar{x}$ for $0 \in A\bar{x}$, i.e., 
 	\begin{align} \label{eq:theorem:MetrSubregOptimalBounds:MetricSub} 
 	(\exists \kappa >0) (\exists \delta >0) (\forall x \in B[\bar{x}; \delta]) \quad \dist \lr{x, A^{-1}0} \leq \kappa \dist \lr{0, Ax}.
 	\end{align}
 	Set $\rho  :=  \max \left\{ \lr{1-\frac{\lambda}{ \frac{\kappa}{\gamma}+1}}^{2}, \lr{ 1 - \lambda \lr{2-\lambda} \frac{1}{1+\frac{\kappa^{2} }{\gamma^{2} }} }   \right\}^{\frac{1}{2}}  $.  Suppose that  $\J_{\gamma A}x \in B[\bar{x}; \delta]$.	Then the following hold. 
 	\begin{enumerate}
 		\item \label{theorem:MetrSubregOptimalBoundszx:MetricSub:rho} $\rho  \in  \left]0,1\right[\,$.
 		\item \label{theorem:MetrSubregOptimalBoundszx:MetricSub:yx}   $\norm{y_{x} - \Pro_{\zer A} \lr{\J_{\gamma A}x} }  \leq \rho  \norm{x -  \Pro_{\zer A} \lr{\J_{\gamma A}x} } $.
 		
 		\item \label{theorem:MetrSubregOptimalBoundszx:MetricSub:barx}  If $\zer A =\{ \bar{x} \}$, then $\norm{y_{x} - \bar{x} } \leq \rho \norm{x - \bar{x} }$.
 		
 		\item \label{theorem:MetrSubregOptimalBoundszx:MetricSub:zx}   Suppose that $\norm{e} \leq \varepsilon \norm{x -z_{x}}$ and that $\eta \varepsilon \in \left[0,1 \right[\,$. Then
 		\begin{align*}
 		\norm{z_{x} - \Pro_{\zer A} \lr{ \J_{\gamma A}x  } }  \leq \frac{\rho  + \eta \varepsilon }{1 - \eta \varepsilon } \norm{x -  \Pro_{\zer A} \lr{\J_{\gamma A}x }  }.
 		\end{align*}
 		In addition, if $\zer A =\{ \bar{x} \}$, then 
 		\begin{align*}
 		\norm{z_{x} -  \bar{x}}  \leq \frac{\rho  + \eta \varepsilon }{1  - \eta \varepsilon } \norm{x -   \bar{x} }.
 		\end{align*}
 	\end{enumerate}

 	\item \label{theorem:MetrSubregOptimalBoundszx:Lipschitz} 
 	Suppose that $A^{-1}$ is Lipschitz continuous at $0$ with modulus $\alpha >0$, i.e., $A^{-1}(0) =\{\bar{x}\}$ and there exists $\tau >0$ such that 
 	\begin{align}  \label{eq:heorem:MetrSubregOptimalBounds:Lipschitz} 
 	\lr{ \forall (w,x) \in \gra A^{-1} \text{ with } w \in B[0;\tau]} \quad \norm{x -\bar{x}} \leq \alpha \norm{w}.
 	\end{align}
 	Set $\rho  :=  \max \left\{ \lr{1-\frac{\lambda}{ \frac{\alpha}{\gamma}+1}}^{2}, \lr{ 1 - \lambda \lr{2-\lambda} \frac{1}{1+\frac{\alpha^{2} }{\gamma^{2} }} }   \right\}^{\frac{1}{2}}  $. Suppose that  $\frac{1}{\gamma} \lr{x-\J_{\gamma A}x} \in B[0; \tau]$.  Then the following hold.
 	\begin{enumerate}
 		\item \label{theorem:MetrSubregOptimalBoundszx:Lipschitz:rho}  $\rho  \in  \left]0,1\right[\,$.
 		\item \label{theorem:MetrSubregOptimalBoundszx:Lipschitz:yx}   $\norm{y_{x} - \bar{x}}  \leq \rho   \norm{x -   \bar{x} }$.
 		
 		\item \label{theorem:MetrSubregOptimalBoundszx:Lipschitz:zx}  
 		Suppose that $\norm{e} \leq \varepsilon \norm{x -z_{x}}$ and that $\eta \varepsilon \in \left[0,1 \right[\,$.   Then
 		\begin{align*}
 		\norm{z_{x} -  \bar{x} }  \leq \frac{\rho  + \eta \varepsilon}{1  -\eta \varepsilon } \norm{x - \bar{x} }.
 		\end{align*}
 	\end{enumerate}  
 \end{enumerate}
\end{theorem}

\begin{proof}
\cref{theorem:MetrSubregOptimalBoundszx:MetricSub}: Because $\J_{\gamma A}x \in B[\bar{x}; \delta]$,  adopt \cref{eq:theorem:MetrSubregOptimalBounds:MetricSub}  with $x = \J_{\gamma A}x$ in the first inequality to derive that 
\begin{align} \label{eq:theorem:MetrSubregOptimalBoundszx:MetricSub}
\norm{\J_{\gamma A}x -\Pro_{\zer A} \lr{ \J_{\gamma A}x} }  =\dist \lr{\J_{\gamma A}x, A^{-1}0}  
  \leq \kappa \dist \lr{0, A \lr{\J_{\gamma A}x}}  
  \leq \frac{\kappa}{\gamma} \norm{x -\J_{\gamma A}x},
\end{align}
where we utilize \cref{corollary:JcA} in the last inequality.
In view of  $\Pro_{\zer A} \lr{ \J_{\gamma A}x} \in \zer A$, employing \cref{eq:theorem:MetrSubregOptimalBoundszx:MetricSub} and applying \cref{theorem:MetrSubregOptimalBounds}\cref{theorem:MetrSubregOptimalBounds:general} with $t =\frac{\kappa}{\gamma}$ and $z= \Pro_{\zer A} \lr{ \J_{\gamma A}x} $, we establish \cref{theorem:MetrSubregOptimalBoundszx:MetricSub:rho} and  \cref{theorem:MetrSubregOptimalBoundszx:MetricSub:yx}.

Notice that if $\zer A =\{ \bar{x} \}$, then $\Pro_{\zer A} \lr{\J_{\gamma A}x} =\bar{x}$. Hence, \cref{theorem:MetrSubregOptimalBoundszx:MetricSub:barx} is clear from \cref{theorem:MetrSubregOptimalBoundszx:MetricSub:yx}.

In addition, according to \cref{cor:fact:FixJcAzerA}, we know that $\J_{\gamma A}$ is $\frac{1}{2}$-averaged and $\Fix  \J_{\gamma A} = \zer A   \neq \varnothing$. Hence, 
combine \cref{theorem:MetrSubregOptimalBoundszx:MetricSub:yx}$\&$\cref{theorem:MetrSubregOptimalBoundszx:MetricSub:barx} with 	\cref{lemma:yxzx}  to guarantee \cref{theorem:MetrSubregOptimalBoundszx:MetricSub:zx}.

\cref{theorem:MetrSubregOptimalBoundszx:Lipschitz}: Employing \cref{corollary:JcA} again, we get that $ \lr{ \J_{\gamma A}x, \frac{1}{\gamma} \lr{x-\J_{\gamma A}x} } \in \gra A$. Taking the assumption, $\frac{1}{\gamma} \lr{x-\J_{\gamma A}x} \in B[0; \tau]$,  and 
\cref{eq:heorem:MetrSubregOptimalBounds:Lipschitz}  into account, we establish that
\begin{align*}
\norm{\J_{\gamma A}x - \bar{x} } \leq \frac{\alpha}{\gamma} \norm{x -\J_{\gamma A}x},
\end{align*}
which, applying
\cref{theorem:MetrSubregOptimalBounds}\cref{theorem:MetrSubregOptimalBounds:general}   with $t =\frac{\alpha}{\gamma}$ and $z=\bar{x}$, ensures \cref{theorem:MetrSubregOptimalBoundszx:Lipschitz:rho} and \cref{theorem:MetrSubregOptimalBoundszx:Lipschitz:yx}. 	 

At last, similarly with the proof of \cref{theorem:MetrSubregOptimalBoundszx:MetricSub:zx} above,   \cref{theorem:MetrSubregOptimalBoundszx:Lipschitz:yx} and \cref{lemma:yxzx}  entail \cref{theorem:MetrSubregOptimalBoundszx:Lipschitz:zx} directly. 
\end{proof}

\begin{remark}  \label{remark:theorem:MetrSubregOptimalBoundszx}
\begin{enumerate}
	\item \label{remark:theorem:MetrSubregOptimalBoundszx:better}  Suppose $\zer A =\{ \bar{x} \}$. 
	If the assumption that $A^{-1}$ is Lipschitz continuous at $0$ with a positive modulus  is strictly stronger than that $A$ is metrically subregular at $\bar{x}$ for $0 \in A\bar{x}$, then \cref{theorem:MetrSubregOptimalBoundszx}\cref{theorem:MetrSubregOptimalBoundszx:MetricSub} is  more interesting than  \cref{theorem:MetrSubregOptimalBoundszx}\cref{theorem:MetrSubregOptimalBoundszx:Lipschitz}.
	\item The idea of \cref{theorem:MetrSubregOptimalBoundszx}\cref{theorem:MetrSubregOptimalBoundszx:Lipschitz:yx}   is essentially presented in  
	\cite[Theorem~3.1]{GuYang2019} on the linear convergence rate of the exact  version of the generalized proximal point algorithm with the relaxation coefficient being a constant in $\mathbb{R}^{n}$.
	 Notice that the proof of \cref{theorem:MetrSubregOptimalBoundszx}\cref{theorem:MetrSubregOptimalBoundszx:Lipschitz:yx}  is closely related to \cref{theorem:MetrSubregOptimalBounds}\cref{theorem:MetrSubregOptimalBounds:leq1}$\&$\cref{theorem:MetrSubregOptimalBounds:geq1}, that  
	\cref{lemma:uvINEQ} is critical to   the proof of \cref{theorem:MetrSubregOptimalBounds}\cref{theorem:MetrSubregOptimalBounds:leq1}, and that   \cref{lemma:uvINEQ}  is inspired by  \cite[Theorem~3.1]{GuYang2019} (see  \cref{remark:uvINEQ} for details). But the proof of \cref{theorem:MetrSubregOptimalBoundszx}\cref{theorem:MetrSubregOptimalBoundszx:Lipschitz:yx}     is more natural and easier to understand than that of \cite[Theorem~3.1]{GuYang2019}. In addition, actually we shall use \cref{theorem:MetrSubregOptimalBoundszx} to investigate the  linear convergence of the inexact version of generalized proximal point algorithms  later.

	\item It is not difficult to see that     \cref{theorem:MetrSubregOptimalBoundszx}\cref{theorem:MetrSubregOptimalBoundszx:Lipschitz:yx}   can actually also be obtained by employing \cref{prop:LipschitzMetricSubregularity} and applying \cref{theorem:MetrSubregOptimalBoundszx}\cref{theorem:MetrSubregOptimalBoundszx:MetricSub:barx} with an extra assumption on the distance from $x$ to $\zer A$.
\end{enumerate}
\end{remark}

%%%%%%%%%%%%%%%%%%%%%%%%%%%%%%%%%%%%%%%%%%%%%%%%%%%%
%%%%%%%%%\section{Krasnosel'ski\v{\i}-Mann iterations}%%%%%%%%%%%%%%%%%%%
%%%%%%%%%%%%%%%%%%%%%%%%%%%%%%%%%%%%%%%%%%%%%%%%%%%%
\section{Inexact Version of the  Non-stationary Krasnosel'ski\v{\i}-Mann iterations} \label{sec:KMannIterations}

In the whole section,  we suppose that $(\alpha_{k})_{k\in \mathbb{N}}$ is in $\left]0,1\right]$, that  $(\forall k \in \mathbb{N})$ $\lambda_{k} \in \left[ 0,\frac{1}{\alpha_{k}}  \right]$,  and that 
\begin{align*}
(\forall k \in \mathbb{N}) \quad T_{k} : \mathcal{H} \to \mathcal{H} \text{ is } \alpha_{k}\text{-averaged with } \cap_{i \in \mathbb{N}}\Fix T_{i} \neq \varnothing.
\end{align*}
Let $x_{0} $ and $\lr{e_{k}}_{k \in \mathbb{N}}$ be in $ \mathcal{H}$ and let $\lr{\eta_{k}}_{k \in \mathbb{N}}$ be in $\mathbb{R}_{+}$.  In this section, we investigate the \emph{inexact non-stationary Krasnosel'ski\v{\i}-Mann iterations} generated by following the iteration scheme
\begin{align}     \label{eq:Tkxk}
\lr{\forall k \in \mathbb{N}} \quad x_{k+1} =(1-\lambda_{k})x_{k} +\lambda_{k} T_{k}x_{k} +\eta_{k}e_{k}.
\end{align}  

Note that the generalized proximal point algorithm studied in this work is actually a special case of the iteration sequence generated by \cref{eq:Tkxk}. Some results obtained in this section will be applied to generalized proximal point algorithms in the following section.

\begin{fact} \label{theorem:KMBasic} {\rm \cite[Theorem~4.3]{OuyangStabilityKMIterations2022}}
	The following statements hold. 
	\begin{enumerate}
		
		\item \label{theorem:KMBasic:induction} $(\forall \bar{x} \in \cap_{k\in \mathbb{N}}\Fix T_{k})$ $(\forall k \in \mathbb{N})$ $\norm{x_{k+1} -\bar{x}} \leq  \norm{x_{0} -\bar{x}} + \sum^{k}_{i =0} \eta_{i}\norm{e_{i}}$.
		
		\item  \label{theorem:KMBasic:sumek} Suppose that $\sum_{k \in \mathbb{N}}   \eta_{k} \norm{e_{k}} <\infty$. 
		Then the following hold. 
		\begin{enumerate}
			
			\item  \label{theorem:KMBasic:sum} $\sum_{k \in \mathbb{N}} \lambda_{k}  \lr{\frac{1}{\alpha_{k}} -\lambda_{k}} \norm{x_{k} -T_{k}x_{k}}^{2} < \infty$.
			
			\item  \label{theorem:KMBasic:l2} Suppose that $\liminf_{k \to \infty}  \lambda_{k}  \lr{\frac{1}{\alpha_{k}} -\lambda_{k}} >0$ $($e.g., $\liminf_{k \to \infty}  \lambda_{k} >0$ and $\limsup_{k \to \infty} \lambda_{k} < \frac{1}{\limsup_{k \to \infty} \alpha_{k}} <\infty$$)$. Then $\sum_{k \in \mathbb{N}}   \norm{x_{k} -T_{k}x_{k}}^{2} < \infty$. Consequently,  $\lim_{k \to \infty} \norm{x_{k} -T_{k}x_{k}}=0$.
			
		\end{enumerate}		
	\end{enumerate}
\end{fact}

\begin{lemma} \label{lemma:xkzCapFix}
	
	Denote by $C:=  \cap_{k\in \mathbb{N}}\Fix T_{k}$. 
 Define 
\begin{align*}
(\forall k \in \mathbb{N}) \quad y_{k} := (1-\lambda_{k})x_{k} +\lambda_{k} T_{k}x_{k}   \text{ and }   \varepsilon_{k} :=  \eta_{k} \norm{e_{k}} \lr{2\norm{y_{k} -\Pro_{C}x_{k}} +\eta_{k} \norm{e_{k}}}.
\end{align*} 
Then  
\begin{align*}
(\forall k \in \mathbb{N}) \quad  \dist^{2} \lr{x_{k+1},  C} \leq \dist^{2} \lr{x_{k},  C} -\lambda_{k}  \lr{\frac{1}{\alpha_{k}} -\lambda_{k}} \norm{x_{k} -T_{k}x_{k}}^{2} +\varepsilon_{k}.
\end{align*}
\end{lemma}

\begin{proof}
Inasmuch as $(\forall k \in \mathbb{N})$ $T_{k}$ is nonexpansive, via	\cite[Proposition~4.23(ii)]{BC2017},    $C=  \cap_{k\in \mathbb{N}}\Fix T_{k}  $ is  closed and convex. So, by \cite[Theorem~3.16]{BC2017}, $(\forall  x \in \mathcal{H})$ $\Pro_{ C} x$ is a well-defined point in $C$.
	
	For every $k \in \mathbb{N}$, applying \cref{eq:fact:lemma:yxzx:z}  in  \cref{fact:lemma:yxzx}  with  $T=T_{k}$, $\alpha =\alpha_{k}$, $x=x_{k}$, $y_{x}=y_{k}$, $z_{x}=x_{k+1}$,  $\eta =\eta_{k}$, $e=e_{k}$, and $\bar{x} =\Pro_{C}x_{k}$ in the second inequality below, we derive that 
	\begin{align*}
 \dist^{2} \lr{x_{k+1},  C}
	~\leq~ &\norm{x_{k+1} - \Pro_{C}x_{k}}^{2}\\
	~\leq~ & \norm{x_{k} - \Pro_{C}x_{k}}^{2} -\lambda_{k}  \lr{\frac{1}{\alpha_{k}} -\lambda_{k}} \norm{x_{k} -T_{k}x_{k}}^{2} +\varepsilon_{k}\\
	~=~& \dist^{2} \lr{x_{k},  C} -\lambda_{k}  \lr{\frac{1}{\alpha_{k}} -\lambda_{k}} \norm{x_{k} -T_{k}x_{k}}^{2} +\varepsilon_{k}.
	\end{align*}

	Altogether, the proof is complete.
\end{proof}

\begin{lemma} \label{Lemma:Tkdistlinear}
	Let $\bar{x} \in \cap_{k\in \mathbb{N}}\Fix T_{k}$. Suppose that $(\forall k \in \mathbb{N})$ $\Id -T_{k}$ is metrically subregular at $\bar{x}$ for 
	$0 \in \lr{\Id -T_{k}}\bar{x}$, i.e., 
	\begin{align}    \label{eq:Lemma:Tkdistlinear:MetricSub}
	(\exists \gamma_{k} >0) (\exists \delta_{k} >0) (\forall x \in B[\bar{x}; \delta_{k}]) \quad \dist \lr{x, \Fix T_{k}} \leq \gamma_{k} \norm{x - T_{k}x}.
	\end{align}
	Suppose that $\delta:=\inf_{k \in \mathbb{N}} \delta_{k} >0$ and that $  \sum_{k \in \mathbb{N}} \eta_{k}\norm{e_{k}} < \delta$.
	Let $0 < \tau \leq \delta - \sum_{k \in \mathbb{N}} \eta_{k}\norm{e_{k}} $ and let $x_{0} \in B[\bar{x}; \tau]$.
	Then the following hold. 
	\begin{enumerate}
		\item \label{Lemma:Tkdistlinear:xk} $(\forall k \in \mathbb{N})$ $ x_{k} \in B[\bar{x}; \delta]$.
		
		\item \label{Lemma:Tkdistlinear:supkappa} Suppose that $(\forall k \in \mathbb{N})$ $C:=\Fix T_{k} \neq \varnothing$, that $\gamma:=\sup_{k \in \mathbb{N}} \gamma_{k} < \infty$, and that   $\liminf_{k \to \infty}  \lambda_{k}  \lr{\frac{1}{\alpha_{k}} -\lambda_{k}} >0$ $($e.g., $\liminf_{k \to \infty}  \lambda_{k} >0$ and $\limsup_{k \to \infty} \lambda_{k} < \frac{1}{\limsup_{k \to \infty} \alpha_{k}} < \infty$$)$. Then $\sum_{k \in \mathbb{N}} \dist^{2} \lr{x_{k}, C} < \infty$. 
		
		Consequently, $ \dist  \lr{x_{k}, C} \to 0$.
	\end{enumerate}
\end{lemma} 

\begin{proof}
	
	\cref{Lemma:Tkdistlinear:xk}: Due to 	\cref{theorem:KMBasic}\cref{theorem:KMBasic:induction} and the assumptions that $x_{0} \in B[\bar{x};\tau]$ and $\tau + \sum_{k \in \mathbb{N}} \eta_{k}\norm{e_{k}} \leq  \delta  $, we easily get \cref{Lemma:Tkdistlinear:xk}.
	
	\cref{Lemma:Tkdistlinear:supkappa}: Based on the assumptions and  \cref{theorem:KMBasic}\cref{theorem:KMBasic:l2},   $\sum_{k \in \mathbb{N}}   \norm{x_{k} -T_{k}x_{k}}^{2} < \infty$. 
	Bearing \cref{Lemma:Tkdistlinear:xk} in mind and for every $k \in \mathbb{N}$, applying  \cref{eq:Lemma:Tkdistlinear:MetricSub} with $x=x_{k}$, we observe that
	\begin{align*}
	\sum_{k \in \mathbb{N}}  \dist^{2} \lr{x_{k}, C} = 	\sum_{k \in \mathbb{N}}  \dist^{2} \lr{x_{k}, \Fix T_{k}} \leq \gamma^{2} \sum_{k \in \mathbb{N}} \norm{x_{k} -T_{k}x_{k}}^{2}  < \infty.
	\end{align*}
	Altogether, the proof is complete.
\end{proof}

 The following result is inspired by the proof of  \cite[Proposition~4.2(iii)]{CombettesFejerAnalysis2001} which is on iteration sequences generated by $\mathcal{T}$-class operators (see \cite[Page~3]{CombettesFejerAnalysis2001} for a detailed definition). 
\begin{proposition} \label{prop:normxk+1xkl2}
	Suppose that $(\forall k \in \mathbb{N})$ $\lambda_{k} \in \left[ 0, \frac{1}{\alpha_{k}}  \right[$, that $\sum_{k \in \mathbb{N}}  \eta_{k}  \norm{ e_{k}} <\infty$, that $\limsup_{k \to \infty} \alpha_{k}  >0$, and that $\limsup_{k \to \infty} \lambda_{k} < \frac{1}{\limsup_{k \to \infty} \alpha_{k}}$. Define $M:=	\limsup_{k \to \infty} \frac{\lambda_{k}}{ \frac{1}{\alpha_{k}} -\lambda_{k}}$. Then the following hold.
	\begin{enumerate}
		\item \label{prop:normxk+1xkl2:M} $M < \infty$.
		\item \label{prop:normxk+1xkl2:LEQ} $\lr{\exists K \in \mathbb{N}}$ $\lr{\forall k \geq K}$ $\norm{x_{k+1} -x_{k}}^{2} \leq 2\ \lr{ M+1}  \lambda_{k}  \lr{\frac{1}{\alpha_{k}} -\lambda_{k}} \norm{x_{k} -T_{k}x_{k}}^{2}  + 2 \eta_{k}^{2} \norm{e_{k}}^{2} $.
		\item \label{prop:normxk+1xkl2:SUM}
		$\sum_{k \in \mathbb{N}} \norm{x_{k+1} -x_{k}}^{2} < \infty$. Consequently, $\norm{x_{k+1} -x_{k}} \to 0$. 	
	\end{enumerate}
\end{proposition}

\begin{proof}
	\cref{prop:normxk+1xkl2:M}: Inasmuch as $\limsup_{k \to \infty} \alpha_{k}  >0$ and $(\forall k \in \mathbb{N})$ $\alpha_{k } \in \left]0,1\right]$, we observe that $\limsup_{k \to \infty} \alpha_{k}  \in \left]0,1\right]$,
	which, connected with $\limsup_{k \to \infty} \lambda_{k} < \frac{1}{\limsup_{k \to \infty} \alpha_{k}}$, ensures that 
	\begin{align*} 
	M=	\limsup_{k \to \infty} \frac{\lambda_{k}}{ \frac{1}{\alpha_{k}} -\lambda_{k}} \leq \frac{ \limsup_{k \to \infty}\lambda_{k}  }{ \frac{1}{\limsup_{k \to \infty} \alpha_{k}} - \limsup_{k \to \infty}\lambda_{k} } < \infty.
	\end{align*}
	
	\cref{prop:normxk+1xkl2:LEQ}: 	Due to  	\cref{prop:normxk+1xkl2:M}, there exists $K \in \mathbb{N}$ such that $\lr{\forall k \geq K}$  $\frac{\lambda_{k}}{ \frac{1}{\alpha_{k}} -\lambda_{k}} \leq M +1$. Using this in the last inequality below, we observe that  
for every $k \geq K$, 
	\begin{align*}
	\norm{x_{k+1} -x_{k}}^{2} &  \stackrel{\cref{eq:Tkxk}}{=} \norm{ \lambda_{k} \lr{T_{k}x_{k} -x_{k}} + \eta_{k} e_{k}}^{2}\\
&	~\leq~ 2 \lambda_{k}^{2} \norm{ T_{k} x_{k}-x_{k}}^{2}  + 2 \eta_{k}^{2} \norm{e_{k}}^{2}\\
&	~=~  2 \frac{\lambda_{k}}{ \frac{1}{\alpha_{k}} -\lambda_{k} } \lambda_{k}  \lr{\frac{1}{\alpha_{k}} -\lambda_{k}} \norm{x_{k} -T_{k}x_{k}}^{2}  + 2 \eta_{k}^{2} \norm{e_{k}}^{2}\\
 & ~	\leq ~ 2\lr{M+1}  \lambda_{k}  \lr{\frac{1}{\alpha_{k}} -\lambda_{k}} \norm{x_{k} -T_{k}x_{k}}^{2}  + 2 \eta_{k}^{2} \norm{e_{k}}^{2}.
	\end{align*}
	
	\cref{prop:normxk+1xkl2:SUM}:
	Combining \cref{prop:normxk+1xkl2:LEQ} with \cref{theorem:KMBasic}\cref{theorem:KMBasic:sum} and the assumption that $\sum_{k \in \mathbb{N}} \eta_{k}   \norm{ e_{k}} <\infty$, we  derive that $\sum_{k \in \mathbb{N}} \norm{x_{k+1} -x_{k}}^{2} < \infty$, which is followed immediately by the last assertion 
	that $\norm{x_{k+1} -x_{k}} \to 0$.
\end{proof}

\cref{theorem:Tkdistlinear} is inspired by \cite[Theorem~3]{LiangFadiliPeyre2016}. In particular,  it generalizes \cite[Theorem~3]{LiangFadiliPeyre2016} by replacing the nonexpansive operator $T$ therein with a sequence of averaged operators $(T_{k})_{k \in \mathbb{N}}$.
In   \cref{theorem:Tkdistlinear} below, we consider the convergence of the sequence $\lr{\dist  \lr{x_{k}, C}  }_{k \in \mathbb{N}}$, where $(x_{k})_{k \in \mathbb{N}}$ is generated by \cref{eq:Tkxk} with $(\forall k \in \mathbb{N})$ $C:=\Fix T_{k}  $.  \cref{theorem:Tkdistlinear} will be applied to deduce a $R$-linear convergence result on generalized proximal point algorithms in the next section.
\begin{proposition} \label{theorem:Tkdistlinear}
	Suppose that $(\forall k \in \mathbb{N})$ $C:=\Fix T_{k} \neq \varnothing$. Let $\bar{x} \in C$. Suppose that $(\forall k \in \mathbb{N})$ $\Id -T_{k}$ is metrically subregular at $\bar{x}$ for 
	$0 \in \lr{\Id -T_{k}}\bar{x}$, i.e., 
	\begin{align} \label{eq:theorem:Tkdistlinear:MetricSub} 
	(\exists \gamma_{k} >0) (\exists \delta_{k} >0) (\forall x \in B[\bar{x}; \delta_{k}]) \quad \dist \lr{x, \Fix T_{k}} \leq \gamma_{k} \norm{x - T_{k}x}.
	\end{align}
	Suppose that $\delta:=\inf_{k \in \mathbb{N}} \delta_{k} >0$ and that $  \sum_{k \in \mathbb{N}} \eta_{k}\norm{e_{k}} < \delta$.
	Let $0 < \tau \leq  \delta - \sum_{k \in \mathbb{N}} \eta_{k}\norm{e_{k}}  $ and let $x_{0} \in B[\bar{x}; \tau]$.
	Define for every $k \in \mathbb{N}$
	\begin{align*}
	&y_{k}:= (1-\lambda_{k})x_{k} +\lambda_{k} T_{k}x_{k}, \quad  \varepsilon_{k} :=  \eta_{k} \norm{e_{k}} \lr{2\norm{y_{k} -\Pro_{C}x_{k}} +\eta_{k} \norm{e_{k}}},\\
	&\beta_{k}:= \frac{\lambda_{k} \lr{\frac{1}{\alpha_{k}} -\lambda_{k}} }{\gamma_{k}^{2}}, \quad \text{and} \quad
	\rho_{k} :=
	\begin{cases}
	1-\beta_{k} \quad &\text{ if } \beta_{k} \leq 1;\\
	\frac{1}{1 + \beta_{k}} \quad &\text{ if } \beta_{k} >1.
	\end{cases}
	\end{align*}
	Then the following assertions hold.
	\begin{enumerate}

		\item  \label{theorem:Tkdistlinear:dist}  $(\forall k \in \mathbb{N})$  $\dist^{2} \lr{x_{k+1}, C} \leq \lr{1-\beta_{k}} \dist^{2} \lr{x_{k}, C} +\varepsilon_{k}$.
		\item \label{theorem:Tkdistlinear:rhok}  $(\forall k \in \mathbb{N})$ $\rho_{k} \in \left[0,1\right]$ and  $\dist^{2} \lr{x_{k+1}, C} \leq \rho_{k} \dist^{2} \lr{x_{k}, C} +\varepsilon_{k}$.
		\item  \label{theorem:Tkdistlinear:prod}  $(\forall k \in \mathbb{N})$ $\dist^{2} \lr{x_{k+1}, C} \leq  \lr{\prod^{k}_{i=0} \rho_{i}  }  \dist^{2} \lr{x_{0}, C} + \sum^{k}_{i=0} \lr{ \prod^{k}_{j=i+1} \rho_{j} } \varepsilon_{i}$. Moreover, the following hold. 
		\begin{enumerate}
			\item \label{theorem:Tkdistlinear:prod:<1}  Suppose that  $\limsup_{k \to \infty} \lr{ \prod^{k}_{i=0}\rho_{i}}^{\frac{1}{k}} <1$ and   $(\forall k \in \mathbb{N})$ $\rho_{k+1} \leq \rho_{k}$. Then $\sum_{k \in \mathbb{N}} \dist^{2} \lr{x_{k}, C} <\infty$. 
			
			Consequently, $\dist  \lr{x_{k}, C} \to 0$.
			\item  \label{theorem:Tkdistlinear:prod:linera} Suppose that $\rho:= \sup_{k \in \mathbb{N}} \rho_{k} <1$. Then 
			\begin{align*}
			(\forall k \in \mathbb{N})\quad \dist^{2} \lr{x_{k+1}, C} \leq \rho^{k} \lr{   \rho  \dist^{2} \lr{x_{0}, C} + \sum^{k}_{i=0}  \frac{ \varepsilon_{i} }{ \rho^{i} } }.  
			\end{align*}
			Consequently, if $\sum_{k \in \mathbb{N}}  \frac{ \varepsilon_{k} }{ \rho^{k} } < \infty$, then $\lr{\dist^{2} \lr{x_{k}, C}}_{ k \in \mathbb{N}}$ converges $R$-linearly  to $0$ .  
		\end{enumerate}
		
	\end{enumerate}
\end{proposition}

\begin{proof}
	\cref{theorem:Tkdistlinear:dist}: Based on \cref{Lemma:Tkdistlinear}\cref{Lemma:Tkdistlinear:xk}, we know that $(\forall k \in \mathbb{N})$ $ x_{k} \in B[\bar{x}; \delta]$.  For every $k \in \mathbb{N}$, apply \cref{eq:theorem:Tkdistlinear:MetricSub}  with  $x =x_{k}$ to deduce that 
	\begin{align}\label{eq:theorem:Tkdistlinear:dist} 
	\dist \lr{x_{k},  C}  = 	\dist \lr{x_{k},  \Fix T_{k}} \leq \gamma_{k}  \norm{x_{k} -T_{k}x_{k}}. 
	\end{align}  
	Apply  \cref{lemma:xkzCapFix} in the first inequality below to get that 
	\begin{align*}
	 \dist^{2} \lr{x_{k+1}, C} 
	~\leq~ &  \dist^{2} \lr{x_{k},  C}  -\lambda_{k}  \lr{\frac{1}{\alpha_{k}} -\lambda_{k}} \norm{x_{k} -T_{k}x_{k}}^{2} +\varepsilon_{k}\\
	\stackrel{\cref{eq:theorem:Tkdistlinear:dist}}{\leq} &   \dist^{2} \lr{x_{k},  C}  - \frac{\lambda_{k} \lr{\frac{1}{\alpha_{k}} -\lambda_{k}} }{\gamma_{k}^{2}} \dist^{2} \lr{x_{k}, C}  +\varepsilon_{k}\\
	~\leq~ & \lr{1 - \frac{\lambda_{k} \lr{\frac{1}{\alpha_{k}} -\lambda_{k}} }{\gamma_{k}^{2}}  } \dist^{2} \lr{x_{k}, C} +\varepsilon_{k},
	\end{align*}
	which guarantees  \cref{theorem:Tkdistlinear:dist}.

	\cref{theorem:Tkdistlinear:rhok}: Note that $(\forall k \in \mathbb{N})$ $\lr{1-\beta_{k}} \leq  \frac{1}{1 + \beta_{k} } \Leftrightarrow 1- \beta^{2}_{k} \leq 1$.
	So we know that $(\forall k \in \mathbb{N})$ $1-\beta_{k} \leq \rho_{k}$. 
	Hence, \cref{theorem:Tkdistlinear:rhok} is clear from  \cref{theorem:Tkdistlinear:dist} above. 
	
	\cref{theorem:Tkdistlinear:prod}: Applying \cref{theorem:Tkdistlinear:rhok}, by induction, we easily  establish that  
	\begin{align}\label{eq:theorem:Tkdistlinear:prod}
	(\forall k \in \mathbb{N}) \quad \dist^{2} \lr{x_{k+1}, C} \leq  \lr{ \prod^{k}_{i=0} \rho_{i} }  \dist^{2} \lr{x_{0}, C} + \sum^{k}_{i=0}  \lr{ \prod^{k}_{j=i+1} \rho_{j} } \varepsilon_{i}.
	\end{align}
For every $k \in \mathbb{N}$,  applying \cref{eq:fact:lemma:yxzx:y} in \cref{fact:lemma:yxzx} with $x=x_{k}$, $y_{x}=y_{k}$, $T =T_{k}$, $\lambda=\lambda_{k}$, $\alpha =\alpha_{k}$, and  $\bar{x} = \Pro_{C}x_{k}$, we observe that
\begin{align*}
\norm{y_{k}- \Pro_{C}x_{k}} \leq \norm{x_{k} -\Pro_{C}x_{k}} = \dist \lr{x_{k}, C},
\end{align*}
which, combined with \cite[Corollary~4.4(iii)]{OuyangStabilityKMIterations2022}, implies that $\lr{ \norm{y_{k}- \Pro_{C}x_{k}}  }_{k \in \mathbb{N}}$ is bounded. This together with the assumption $  \sum_{k \in \mathbb{N}} \eta_{k}\norm{e_{k}} < \delta <\infty$ necessitates that $\sum_{k \in \mathbb{N}} \varepsilon_{k} < \infty$.

	\cref{theorem:Tkdistlinear:prod:<1}:   As a consequence of 
	\cref{lemma:cauchyproduct}\cref{lemma:cauchyproduct:chi}$\&$\cref{lemma:cauchyproduct:sum}, our assumptions imply that 
	\begin{align*}
	\sum_{k \in \mathbb{N}} \lr{ \prod^{k}_{i=0} \rho_{i}} <\infty   \quad \text{and} \quad \sum_{k \in \mathbb{N}} \sum^{k}_{i=0}  \lr{ \prod^{k}_{j=i+1} \rho_{j}} \varepsilon_{i} < \infty.
	\end{align*}
	This combined with \cref{eq:theorem:Tkdistlinear:prod}  ensures that $\sum_{k \in \mathbb{N}} \dist^{2} \lr{x_{k}, C} < \infty$, which forces that $\dist  \lr{x_{k}, C}  \to 0$.

	\cref{theorem:Tkdistlinear:prod:linera}: This is immediate from \cref{eq:theorem:Tkdistlinear:prod}. 
\end{proof}

In \cref{Corollary:Tkdistlinear}\cref{Corollary:Tkdistlinear:ek=0:general} below, by applying \cref{theorem:Tkdistlinear}\cref{theorem:Tkdistlinear:dist}, we deduce the main result of \cite[Theorem~4.9]{OuyangStabilityKMIterations2022} again. 
\begin{corollary} \label{Corollary:Tkdistlinear}
	Suppose that $(\forall k \in \mathbb{N})$ $C:=\Fix T_{k} \neq \varnothing$ and that $(\forall k \in \mathbb{N})$ $e_{k} \equiv 0$ and $\eta_{k} \equiv 0$ in \cref{eq:Tkxk}. Let $\bar{x} \in C$. Suppose that $(\forall k \in \mathbb{N})$ $\Id -T_{k}$ is metrically subregular at $\bar{x}$ for 
	$0 \in \lr{\Id -T_{k}}\bar{x}$, i.e., 
	\begin{align*} 
	(\exists \gamma_{k} >0) (\exists \delta_{k} >0) (\forall x \in B[\bar{x}; \delta_{k}]) \quad \dist \lr{x, \Fix T_{k}} \leq \gamma_{k} \norm{x - T_{k}x}.
	\end{align*}
	Suppose that $\delta:=\inf_{k \in \mathbb{N}} \delta_{k} >0$. 
	Let  $x_{0} \in B[\bar{x}; \delta]$.
	Define 
	\begin{align*}
(\forall k \in \mathbb{N}) \quad 	\beta_{k}:= \frac{\lambda_{k} \lr{\frac{1}{\alpha_{k}} -\lambda_{k}} }{\gamma_{k}^{2}} \quad \text{and} \quad
	\rho_{k} := 1-\beta_{k}.
	\end{align*}
	Then the following hold.
	\begin{enumerate}
		\item  \label{Corollary:Tkdistlinear:ek=0:general} $(\forall k \in \mathbb{N})$   $  \rho_{k}  \in \left[0,1\right]$ and  $\dist  \lr{x_{k+1}, C} \leq \rho_{k}^{\frac{1}{2}} \dist  \lr{x_{k}, C}  $.
		Consequently,  $\lim_{k \to \infty}  \dist \lr{x_{k}, C}$ exists. 
		
		\item  \label{Corollary:Tkdistlinear:ek=0:linera} Suppose that $0< \underline{\lambda}:=\inf_{k \in \mathbb{N}} \lambda_{k} \leq \overline{\lambda}:=\sup_{k \in \mathbb{N}}  \lambda_{k} < \frac{1}{\alpha}$ where $\alpha:=\sup_{k \in \mathbb{N}}  \alpha_{k} >0 $ and that $\gamma :=\sup_{k \in \mathbb{N}}  \gamma_{k}  \in \mathbb{R}_{++}$.
		Define $\rho:=\sup_{k \in \mathbb{N}}  \rho_{k}$.
Then there exists $\hat{x} \in C$  such that  
\begin{align}  \label{eq:Corollary:Tkdistlinear:ek=0:linera}
(\forall k \in \mathbb{N}) \quad \norm{x_{k} -\hat{x}} \leq 2 \rho^{\frac{k}{2}} \dist \lr{x_{0}, C}.
\end{align}
Consequently, $\lr{x_{k}}_{k \in \mathbb{N}}$ converges $R$-linearly to a point $\hat{x} \in C$.
	\end{enumerate} 
\end{corollary}

\begin{proof}
We uphold the notation used in \cref{theorem:Tkdistlinear} above. 	Notice that if $(\forall k \in \mathbb{N})$ $e_{k} \equiv 0$ and $\eta_{k} \equiv 0$, then, by definition, $(\forall k \in \mathbb{N})$ $\varepsilon_{k} \equiv 0$  in \cref{theorem:Tkdistlinear}.  Then \cref{theorem:Tkdistlinear}\cref{theorem:Tkdistlinear:dist} forces that $(\forall k \in \mathbb{N})$  $1 - \beta_{k} \in \mathbb{R}_{+}$, that is, $\beta_{k} \leq 1$. Hence, we have that $(\forall k \in \mathbb{N})$ $\rho_{k} = 1- \beta_{k} \in \left[0,1\right]$  in \cref{theorem:Tkdistlinear}.
	
	\cref{Corollary:Tkdistlinear:ek=0:general}: According to our analysis above, this is immediate from \cref{theorem:Tkdistlinear}\cref{theorem:Tkdistlinear:rhok}.

	\cref{Corollary:Tkdistlinear:ek=0:linera}: Because $(\forall k \in \mathbb{N})$ $\rho_{k} = 1- \beta_{k} \in \left[0,1\right]$ and $0 <  \underline{\lambda} \leq \overline{\lambda} < \frac{1}{\alpha} < \infty$, we observe that 
	\begin{align}\label{eq:Corollary:Tkdistlinear:ek=0:linera:rho}
	\rho =\sup_{k \in \mathbb{N}}  \rho_{k} = 1- \inf_{k \in \mathbb{N}} \frac{\lambda_{k} \lr{\frac{1}{\alpha_{k}} -\lambda_{k}} }{\gamma_{k}^{2}} \leq  1- \frac{ \underline{\lambda} \lr{ \frac{1}{ \alpha}  - \overline{\lambda}  } }{\gamma} \in \left[0, 1\right[\,.
	\end{align}
	Notice that \cite[Proposition~4.23(ii)]{BC2017} implies that $C$ is nonempty closed and convex and that our assumptions necessitate $(\forall k \in \mathbb{N})$ $0 <  \inf_{i \in \mathbb{N}} \lambda_{i}  \leq \lambda_{k} \leq  \sup_{i \in \mathbb{N}}  \lambda_{i} < \frac{1}{\sup_{i \in \mathbb{N}}  \alpha_{i} } \leq \frac{1}{\alpha_{k}}$. Hence, 
\cref{eq:Corollary:Tkdistlinear:ek=0:linera:rho} combined with  \cref{Corollary:Tkdistlinear:ek=0:general} above leads to
\begin{align*}
(\forall k \in \mathbb{N}) \quad \dist  \lr{x_{k+1}, C} \leq \rho^{\frac{1}{2}} \dist  \lr{x_{k}, C},
\end{align*}
which, connecting with \cite[Theorem~4.9(i)]{OuyangStabilityKMIterations2022} and \cite[Theorem~5.12]{BC2017}, ensures 
 \cref{eq:Corollary:Tkdistlinear:ek=0:linera}. 
\end{proof}

%%%%%%%%%%%%%%%%%%%%%%%%%%%%%%%%%%%%%%%%%%%%%%%%%%%%
%%%%%%%%%\section{Generalized Proximal Point Algorithms}%%%%%%%%%%%
%%%%%%%%%%%%%%%%%%%%%%%%%%%%%%%%%%%%%%%%%%%%%%%%%%%%

\section{Generalized Proximal Point Algorithms} \label{sec:GPPA}
Throughout this section, $A :\mathcal{H} \to 2^{\mathcal{H}}$ is maximally monotone with $\zer A \neq \varnothing$ and 
\begin{align} \label{eq:GPPA}
(\forall k \in \mathbb{N}) \quad x_{k+1} = \lr{1-\lambda_{k}}x_{k} +\lambda_{k} \J_{c_{k} A}x_{k} +\eta_{k}e_{k}, 
\end{align}
where $x_{0} \in \mathcal{H}$ is the \emph{initial point} and $(\forall k \in \mathbb{N})$ $\lambda_{k} \in \left[0,2\right]$ and $\eta_{k} \in \mathbb{R}_{+}$ are the \emph{relaxation coefficients},  $c_{k} \in \mathbb{R}_{++}$ is the \emph{regularization coefficient},  and $e_{k} \in \mathcal{H}$ is the \emph{error term}. 

Generalized proximal point algorithms generate the iteration sequence by conforming to the scheme  \cref{eq:GPPA}. The classic proximal point algorithm generates the iteration sequence by following \cref{eq:GPPA} with $(\forall k \in \mathbb{N})$ $\lambda_{k} \equiv 1$, $e_{k} \equiv 0$, and $\eta_{k} \equiv 0$.
In this section, we investigate the linear convergence of generalized proximal point algorithms for solving the monotone inclusion problem, i.e., finding a zero of the associated monotone operator.

\subsection{Auxiliary results} \label{section:exactGPPA}

\cref{theorem:exactPAlinear} is   motivated by \cite[Theorem~3.1]{Leventhal2009} and  \cite[Theorem~3.1]{ShenPan2016}. In particular,   \cref{theorem:exactPAlinear}\cref{theorem:exactPAlinear:1:leq} and  \cref{theorem:exactPAlinear}\cref{theorem:exactPAlinear:lambdakK}\textcolor{blue}{i.}\,reduce to \cite[Theorem~3.1]{Leventhal2009} and  \cite[Theorem~3.1]{ShenPan2016}, respectively,  when $(\forall k \in \mathbb{N})$ $c_{k} \equiv c \in \mathbb{R}_{++}$. 
\cref{theorem:exactPAlinear}\cref{theorem:exactPAlinear:1} works on the local convergence of the exact version of the proximal point algorithm.
Note that  if we restrict $(\forall k \in \mathbb{N})$ $\lambda_{k} \equiv 1$ in \cref{theorem:exactPAlinear}\cref{theorem:exactPAlinear:lambdakK},    the convergence rate in \cref{theorem:exactPAlinear}\cref{theorem:exactPAlinear:1K} is better than that of   \cref{theorem:exactPAlinear}\cref{theorem:exactPAlinear:lambdakK} since $\frac{1}{  1 + \frac{c_{k}^{2}}{\kappa^{2}}  } = 1- \frac{1}{1 + \frac{\kappa^{2}}{c_{k}^{2}} } \leq 1 - \frac{1}{\lr{ 1+ \frac{\kappa}{c_{k}}  }^{2}} $.
\begin{proposition} \label{theorem:exactPAlinear}
	Suppose that $(\forall k \in \mathbb{N})$ $e_{k} \equiv 0$ and $\eta_{k} \equiv 0$ in \cref{eq:GPPA}. 	Let $\bar{x}$ be in $\zer A$. Suppose that $A$ is metrically subregular at $\bar{x}$ for $0 \in A\bar{x}$, i.e., 
	\begin{align} \label{eq:theorem:exactPAlinear:MetricSub} 
	(\exists \kappa >0) (\exists \delta >0) (\forall x \in B[\bar{x}; \delta]) \quad \dist \lr{x, A^{-1}0} \leq \kappa \dist \lr{0, Ax}.
	\end{align}
Set $c:=\inf_{k \in \mathbb{N}} c_{k}$.		Then the following assertions hold.
	\begin{enumerate}
		\item \label{theorem:exactPAlinear:1} Suppose that $x_{0} \in B[\bar{x}; \delta] $ and   that $(\forall k \in \mathbb{N})$ $\lambda_{k} \equiv 1$. Then 
		\begin{enumerate}
			\item \label{theorem:exactPAlinear:1:leq}  $(\forall k \in \mathbb{N})$ $ \dist \lr{  x_{k+1}, \zer A}  \leq \frac{1}{ \sqrt{1 + \frac{c_{k}^{2}}{\kappa^{2}} } } \dist \lr{x_{k}, \zer A }$;
			\item \label{theorem:exactPAlinear:1:Rlinear}  if $c  >0$, then $\lr{x_{k}}_{k \in \mathbb{N}}$ converges $R$-linearly to a point $\hat{x} \in \zer A$.
		\end{enumerate}

		\item \label{theorem:exactPAlinear:xktobarx} Suppose that $x_{k} \to \bar{x}$.   Then the following hold. 
		\begin{enumerate}
			\item \label{theorem:exactPAlinear:lambdakK} Set $\lr{\forall k \in \mathbb{N}} $ $ \rho_{k} := \lr{ 1 - \lambda_{k} \lr{2- \lambda_{k}} \frac{1}{ \lr{1+ \frac{\kappa}{c_{k}}}^{2} }}^{\frac{1}{2}}$.  Then 
			\begin{enumerate}
				\item \label{theorem:exactPAlinear:lambdakK:leq} there exists $K \in \mathbb{N}$ such that $	\lr{\forall k \geq K} $ $\dist \lr{x_{k+1}, \zer A}  \leq  \rho_{k} \lr{x_{k}, \zer A}$;
				\item \label{theorem:exactPAlinear:lambdakK:Rlinear} if    $c >0$ and $0< \underline{\lambda}:=\inf_{k \in \mathbb{N}} \lambda_{k} \leq \overline{\lambda}:=\sup_{k \in \mathbb{N}}  \lambda_{k} <  2 $, then		$\lr{x_{k}}_{k \in \mathbb{N}}$ converges $R$-linearly to a point $\hat{x} \in \zer A$.
			\end{enumerate}

			\item \label{theorem:exactPAlinear:1K} Suppose that $(\forall k \in \mathbb{N})$ $\lambda_{k} \equiv 1$. Then
			\begin{enumerate}
				\item there exists $K \in \mathbb{N}$ such that $	\lr{\forall k \geq K}$ $  \dist \lr{  x_{k+1}, \zer A}  \leq \frac{1}{ \sqrt{1 + \frac{c_{k}^{2}}{\kappa^{2}} } } \dist \lr{x_{k}, \zer A }$;
				\item if $c  >0$, then $\lr{x_{k}}_{k \in \mathbb{N}}$ converges $R$-linearly to a point $\hat{x} \in \zer A$.
			\end{enumerate}  

		\end{enumerate} 
	\end{enumerate}
	
\end{proposition}

\begin{proof}
	\cref{theorem:exactPAlinear:1}:  For every $k \in \mathbb{N}$, applying \cref{lemma:resolvents:yxzx} with $x=x_{k}$, $y_{x}=x_{k+1}$, $\lambda =\lambda_{k}$, and $\gamma =c_{k}$, we know that $\norm{x_{k+1} -\bar{x}} \leq \norm{x_{k} -\bar{x}}$; and employing 	\cref{lemma:JGammaAFix} with $x=x_{k}$, $\gamma =c_{k}$, and $z=\bar{x}$, we observe that $\norm{\J_{c_{k} A}x_{k} -\bar{x}} \leq \norm{x_{k} -\bar{x}}$. Combine these results with $x_{0} \in B[\bar{x}; \delta] $ to deduce that 	
%	Because $x_{0} \in B[\bar{x}; \delta] $, for every $k \in \mathbb{N}$, applying \cref{lemma:resolvents:yxzx} with $x=x_{k}$, $y_{x}=x_{k+1}$, $\lambda =\lambda_{k}$, and $\gamma =c_{k}$ and employing 	\cref{lemma:JGammaAFix} with $x=x_{k}$, $\gamma =c_{k}$, and $z=\bar{x}$, we observe, by induction,  that
	\begin{align} \label{eq:theorem:exactPAlinear}
	x_{k} \in B[\bar{x}; \delta]  \quad \text{and} \quad  \J_{c_{k} A}x_{k} \in B[\bar{x}; \delta].
	\end{align}
	
	Let $k \in \mathbb{N}$. 
	Taking \cref{eq:theorem:exactPAlinear}  into account and applying	\cref{lemma:JgammaAINEQ}   with $x =x_{k}$ and $\gamma =c_{k}$, we  obtain that 
	\begin{align*}    
	\dist \lr{  x_{k+1}, \zer A} = \dist \lr{ \J_{c_{k} A}x_{k},  \zer A } \leq \frac{1}{ \sqrt{1 + \frac{c_{k}^{2}}{\kappa^{2}} } } \dist \lr{x_{k}, \zer A }.
	\end{align*}
	
	Suppose that $c =\inf_{k \in \mathbb{N}} c_{k} >0$. Then   $\rho:=\sup_{k \in \mathbb{N}} \frac{1}{ \sqrt{1 + \frac{c_{k}^{2}}{\kappa^{2}} } }  = \frac{1}{ \sqrt{1 + \frac{c^{2}}{\kappa^{2}} } }  \in \left[0,1\right[\,$ and  
	\begin{align} \label{eq:theorem:exactPAlinear:1:rho}
	(\forall k \in \mathbb{N}) \quad \dist \lr{  x_{k+1}, \zer A}  \leq \rho \dist \lr{x_{k}, \zer A }.
	\end{align} 
	Notice that, via \cite[Proposition~23.39]{BC2017} and by virtue of the maximal monotonicity of $A$, $\zer A$ is closed and convex. 
	 Similarly with the proof of \cref{Corollary:Tkdistlinear}\cref{Corollary:Tkdistlinear:ek=0:linera},  combining \cref{eq:theorem:exactPAlinear:1:rho} with \cite[Theorem~5.6(i)]{OuyangStabilityKMIterations2022} and \cite[Theorem~5.12]{BC2017}, we obtain that $\lr{x_{k}}_{k \in \mathbb{N}}$ converges $R$-linearly to a point $\hat{x} \in \zer A$.  
	
	\cref{theorem:exactPAlinear:xktobarx}: Because $x_{k} \to \bar{x}$,   we know that there exists $K \in \mathbb{N}$ such that  $\lr{\forall k \geq K} $ $	x_{k} \in B[\bar{x}; \delta]$, which, connected with \cref{lemma:JGammaAFix}, ensures that 
	\begin{align}\label{eq:theorem:exactPAlinear:K}
	\lr{\forall k \geq K}  \quad  \J_{c_{k} A}x_{k} \in B[\bar{x}; \delta].
	\end{align}

	\cref{theorem:exactPAlinear:lambdakK}: Let $k \in \mathbb{N}$ such that $k \geq K$.   Bearing \cref{eq:theorem:exactPAlinear:K} in mind and applying  \cref{theorem:MetrSubregUpperBound}\cref{theorem:MetrSubregUpperBound:yx} with $x=x_{k}$, $y_{x} =x_{k+1}$, $\gamma =c_{k}$, $\lambda=\lambda_{k}$,  and  $\rho =\rho_{k}$,  we deduce that 
	\begin{align} \label{eq:theorem:exactPAlinear:lambdakK:i}
	\dist \lr{  x_{k+1}, \zer A} \leq \norm{x_{k+1} - \Pro_{\zer A}x_{k}} \leq   \rho_{k} \norm{x_{k} - \Pro_{\zer A}x_{k}}=  \rho_{k} \dist \lr{x_{k}, \zer A }.
	\end{align}
	
	Suppose that $0< \underline{\lambda} =\inf_{k \in \mathbb{N}} \lambda_{k} \leq \overline{\lambda} =\sup_{k \in \mathbb{N}}  \lambda_{k} <  2 $  and $c =\inf_{k \in \mathbb{N}} c_{k} >0$. Then $\rho:= \sup_{k \in \mathbb{N}} \rho_{k} \leq \lr{ 1 - \underline{\lambda }  \lr{2- \overline{\lambda } } \frac{1}{ \lr{1+ \frac{\kappa}{c }}^{2} }}^{\frac{1}{2}} \in \left[0,1\right[\,$. Moreover, due to \cref{eq:theorem:exactPAlinear:lambdakK:i}, 
	\begin{align*}
	\lr{\forall k \geq K} \quad \dist \lr{  x_{k+1}, \zer A} \leq \rho^{k-K+1}  \dist \lr{x_{K}, \zer A },
	\end{align*}
	which, combined with \cite[Theorem~5.6(i)]{OuyangStabilityKMIterations2022} and \cite[Theorem~5.12]{BC2017}, guarantees that 	$\lr{x_{k}}_{k \in \mathbb{N}}$ converges $R$-linearly to a point $\hat{x} \in \zer A$.

	\cref{theorem:exactPAlinear:1K}: Invoking  \cref{eq:theorem:exactPAlinear:K} and employing almost the same arguments used in the proof of  \cref{theorem:exactPAlinear:1:leq} and \cref{theorem:exactPAlinear:lambdakK}\textcolor{blue}{ii.}\,above, we obtain \cref{theorem:exactPAlinear:1K}. 
\end{proof}

\begin{corollary} \label{corollary:ckxkxk+1to0}
 Suppose that	$\sum_{k \in \mathbb{N}}   \eta_{k} \norm{e_{k}} <\infty$. Set $c:=\inf_{k \in \mathbb{N}} c_{k} $. Then the following hold.   
	\begin{enumerate}
		\item  \label{corollary:ckxkxk+1to0:Jck} Suppose that $0 < \liminf_{k \to \infty}  \lambda_{k} \leq \limsup_{k \to \infty} \lambda_{k} < 2$. Then $\sum_{k \in \mathbb{N}}   \norm{x_{k} -\J_{ c_{k} A}x_{k}}^{2} < \infty$. Moreover, if $c  >0$, then $\frac{1}{c_{k}} \lr{x_{k} - \J_{ c_{k} A}x_{k} }\to 0$.  
		\item  \label{corollary:ckxkxk+1to0:xk+1} 
		Suppose that  $\sup_{k \in \mathbb{N}} \lambda_{k} <2$. Then $\sum_{k \in \mathbb{N}}   \norm{x_{k} -x_{k+1}}^{2} < \infty$. Moreover, if $c  >0$ and $\lambda := \inf_{k \in \mathbb{N}} \lambda_{k} >0$, then $\frac{1}{c_{k}} \lr{x_{k} - \J_{ c_{k} A}x_{k} }\to 0$.  
		
	\end{enumerate}
\end{corollary}
	
\begin{proof}
	According to \cref{cor:fact:FixJcAzerA},  $(\forall k \in \mathbb{N})$ $\J_{c_{k} A}$ is $\frac{1}{2}$-averaged operator and   $   \Fix \J_{c_{k} }A =\zer A \neq \varnothing$.
	
	\cref{corollary:ckxkxk+1to0:Jck}:  Applying \cref{theorem:KMBasic}\cref{theorem:KMBasic:l2}  with $(\forall k \in \mathbb{N})$ $T_{k} = \J_{ c_{k} A}$ and $\alpha_{k} =\frac{1}{2}$, we deduce $\sum_{k \in \mathbb{N}}   \norm{x_{k} -\J_{ c_{k} A}x_{k}}^{2} < \infty$, which forces $\norm{x_{k} -\J_{ c_{k} A}x_{k} } \to 0 $. 
	
	In addition, if  $c=\inf_{k \in \mathbb{N}} c_{k} >0$, then 
	\begin{align*}
	\norm{\frac{1}{c_{k}} \lr{x_{k} - \J_{ c_{k} A}x_{k} }} \leq \frac{1}{c} \norm{x_{k} -\J_{ c_{k} A}x_{k} }  \to 0.
	\end{align*}

\cref{corollary:ckxkxk+1to0:xk+1}: 	 It is not difficult to verify that $\sup_{k \in \mathbb{N}} \lambda_{k} <2$ if and only if $
 \limsup_{k \to \infty} \lambda_{k} < 2$ and $(\forall k \in \mathbb{N})$ $ \lambda_{k} <2$.
Apply  \cref{prop:normxk+1xkl2}\cref{prop:normxk+1xkl2:SUM} with $(\forall k \in \mathbb{N})$ $T_{k} = \J_{ c_{k} A}$ and $\alpha_{k} =\frac{1}{2}$ to yield that $\sum_{k \in \mathbb{N}}   \norm{x_{k+1} - x_{k}}^{2} < \infty$ and $ \norm{x_{k+1} - x_{k}} \to 0$.
Suppose that   $c=\inf_{k \in \mathbb{N}} c_{k} >0$ and $\lambda = \inf_{k \in \mathbb{N}} \lambda_{k} >0$. Based on \cref{eq:GPPA},
\begin{align*}
 \norm{ \frac{1}{c_{k}} \lr{x_{k} - \J_{ c_{k} A}x_{k} } }= \norm{\frac{1}{c_{k} \lambda_{k}} \lr{x_{k} -x_{k+1} + \eta_{k}e_{k}}} \leq \frac{1}{c \lambda} \lr{  \norm{x_{k} -x_{k+1} } + \eta_{k}\norm{e_{k}} }\to 0.
\end{align*}

Altogether, the proof is complete.
\end{proof}

\begin{proposition} \label{proposition:Axkweak}
	Suppose that $\sum_{k \in \mathbb{N}} \eta_{k} \norm{e_{k}} < \infty$ and that $\sum_{k \in \mathbb{N}} \abs{1-\frac{c_{k+1}}{c_{k}}} < \infty$. Then the following hold.
	\begin{enumerate}
		\item  \label{proposition:Axkweak:lim} 
		$\lim_{k \to \infty} \norm{x_{k} -\J_{c_{k} A}x_{k}}$ exists. 
		\item  \label{proposition:Axkweak:xkJck} Assume that $\sum_{k \in \mathbb{N}} \lambda_{k} \lr{2-\lambda_{k}} = \infty$.
		Then $x_{k} -\J_{c_{k} A}x_{k} \to 0 $.
	\end{enumerate}
\end{proposition}

\begin{proof}	
	
	\cref{proposition:Axkweak:lim}: Because $\sum_{k \in \mathbb{N}} \abs{1-\frac{c_{k+1}}{c_{k}}} < \infty$ and  $\sum_{k \in \mathbb{N}} \eta_{k} \norm{e_{k}} < \infty$, due to	\cite[Lemma~5.2(ii)]{OuyangStabilityKMIterations2022} and \cite[Lemma~5.31]{BC2017}, we know that $\lim_{k \to \infty} \norm{x_{k} -\J_{c_{k} A}x_{k}} $ exists in $\mathbb{R}_{+}$.

	\cref{proposition:Axkweak:xkJck}: 	Applying \cref{theorem:KMBasic}\cref{theorem:KMBasic:sum} with $(\forall k \in \mathbb{N})$ $T_{k} = \J_{ c_{k} A}$ and $\alpha_{k} =\frac{1}{2}$ and employing
 the assumption $\sum_{k \in \mathbb{N}} \lambda_{k} \lr{2-\lambda_{k}} = \infty$, we   establish that $\liminf_{k \to \infty} \norm{x_{k} -\J_{c_{k} A}x_{k}}  =0$. 
	This as well as \cref{proposition:Axkweak:lim} yields $x_{k} -\J_{c_{k} A}x_{k} \to 0 $. 	
\end{proof}

   \cref{theorem:Axkweakconverge} generalizes \cite[Theorem~4.5]{TaoYuan2018} by replacing the constant $\lambda \in \left]0,2\right[$ therein with a sequence $(\lambda_{k})_{k \in \mathbb{N}}$ in $\left[0,2 \right]$. To do this generalization, we can require that $0 < \inf_{k \in \mathbb{N}} \lambda_{k}\leq \sup_{k \in \mathbb{N}} \lambda_{k} <2$ like  \cite[Theorem~3]{EcksteinBertsekas1992} or that 
$\sum_{k \in \mathbb{N}} \lambda_{k} \lr{2-\lambda_{k}} = \infty$ and $\sum_{k \in \mathbb{N}} \abs{1-\frac{c_{k+1}}{c_{k}}} < \infty$ like our \cref{theorem:Axkweakconverge}\cref{theorem:Axkweakconverge:sum}.    In \cite[Theorem~3]{EcksteinBertsekas1992}, to obtain the required weak convergence of the sequence generated by 	\cref{eq:GPPA} with $(\forall k \in \mathbb{N})$ $\eta_{k} =\lambda_{k}$, the authors require that $0 < \inf_{k \in \mathbb{N}} \lambda_{k}\leq \sup_{k \in \mathbb{N}} \lambda_{k} <2$, which is stronger than our assumption $\sum_{k \in \mathbb{N}} \lambda_{k} \lr{2-\lambda_{k}} = \infty$ in \cref{theorem:Axkweakconverge}, but our assumption $\sum_{k \in \mathbb{N}} \abs{1-\frac{c_{k+1}}{c_{k}}} < \infty$ in \cref{theorem:Axkweakconverge}\cref{theorem:Axkweakconverge:sum}  is not needed for \cite[Theorem~3]{EcksteinBertsekas1992}.

\begin{theorem} \label{theorem:Axkweakconverge}
	Suppose that   $\sum_{k \in \mathbb{N}} \eta_{k} \norm{e_{k}} < \infty$  and $\inf_{k \in \mathbb{N}} c_{k} >0$. 
	Then the following assertions hold.
	\begin{enumerate}
		\item \label{theorem:Axkweakconverge:lim} Suppose that $ \lim_{k \to \infty} \norm{x_{k} -\J_{c_{k} A}x_{k}} =0$. Then  $\lr{x_{k}}_{k \in \mathbb{N}}$ converges weakly to a point in $\zer A$.
		\item \label{theorem:Axkweakconverge:sum} Suppose that 
		$\sum_{k \in \mathbb{N}} \lambda_{k} \lr{2-\lambda_{k}} = \infty$ and $\sum_{k \in \mathbb{N}} \abs{1-\frac{c_{k+1}}{c_{k}}} < \infty$.   Then  $\lr{x_{k}}_{k \in \mathbb{N}}$ converges weakly to a point in $\zer A$.
	\end{enumerate}
\end{theorem}

\begin{proof}
	\cref{theorem:Axkweakconverge:lim}:	Taking   $ \lim_{k \to \infty} \norm{x_{k} -\J_{c_{k} A}x_{k}} =0$,  $\inf_{k \in \mathbb{N}} c_{k} >0$, and \cite[Proposition~2.16(i)]{OuyangWeakStrongGPPA2021} into account, we know that  $\Omega \lr{ \lr{x_{k}}_{k \in \mathbb{N}} } \subseteq \zer A$.  
	On the other hand, due to \cite[Lemma~5.1(iii)(b)]{OuyangStabilityKMIterations2022}, we have that $\lr{ \forall z \in  \zer A } $ $\lim_{k \to \infty} \norm{x_{k} -z}$ exists in $\mathbb{R}_{+}$.
These results combined with   \cite[Lemma~2.47]{BC2017} entail that $\lr{x_{k}}_{k \in \mathbb{N}}$ converges weakly to a point in $\zer A$.
	
	\cref{theorem:Axkweakconverge:sum}:      Combine our assumption with  \cref{proposition:Axkweak}\cref{proposition:Axkweak:xkJck} to get  that $\lim_{k \to \infty} \norm{x_{k} -\J_{c_{k} A}x_{k}} =0$.
	Hence, the desired weak convergence is clear from \cref{theorem:Axkweakconverge:lim} above. 
\end{proof}
The convergence result of  \cref{corolary:Axkweakconverge} is consistent with that of  \cite[Theorem~3]{EcksteinBertsekas1992} except that 
the assumption   $0 < \inf_{k \in \mathbb{N}} \lambda_{k}\leq \sup_{k \in \mathbb{N}} \lambda_{k} <2$ in  \cite[Theorem~3]{EcksteinBertsekas1992} is replaced by  $0 <  \liminf_{k \to \infty} \lambda_{k} \leq  \limsup_{k \to \infty} \lambda_{k} < 2$ in \cref{corolary:Axkweakconverge}.
\begin{corollary} \label{corolary:Axkweakconverge}
	Suppose that   $\sum_{k \in \mathbb{N}} \eta_{k} \norm{e_{k}} < \infty$,
	that $\inf_{k \in \mathbb{N}} c_{k} >0$,
	and that $0 <  \liminf_{k \to \infty} \lambda_{k} \leq  \limsup_{k \to \infty} \lambda_{k} < 2$.
	Then $\lr{x_{k}}_{k \in \mathbb{N}}$ converges weakly to a point in $\zer A$.
\end{corollary}

\begin{proof}
	Clearly,  the assumption $0 <  \liminf_{k \to \infty} \lambda_{k} \leq  \limsup_{k \to \infty} \lambda_{k} < 2$ entails that  $\sum_{k \in \mathbb{N}} \lambda_{k} \lr{2-\lambda_{k}} = \infty$.
		Moreover, apply \cref{theorem:KMBasic}\cref{theorem:KMBasic:sum} with $(\forall k \in \mathbb{N})$ $T_{k} = \J_{ c_{k} A}$ and $\alpha_{k} =\frac{1}{2}$ to get that  
	\begin{align*}
	\sum_{k \in \mathbb{N}} \lambda_{k}  \lr{2 -\lambda_{k}} \norm{x_{k} -\J_{c_{k} A}x_{k}}^{2} <\infty,
	\end{align*}
	which, combined with $0 <  \liminf_{k \to \infty} \lambda_{k} \leq  \limsup_{k \to \infty} \lambda_{k} < 2$, guarantees that 
	$\sum_{k \in \mathbb{N}}   \norm{x_{k} -\J_{c_{k} A}x_{k}}^{2} <\infty$ and hence, $ \lim_{k \to \infty} \norm{x_{k} -\J_{c_{k} A}x_{k}} =0$. 
	
	Therefore, by \cref{theorem:Axkweakconverge}\cref{theorem:Axkweakconverge:lim}, we obtain the required weak convergence.
\end{proof}

 \subsection{Linear convergence of generalized proximal point algorithms} \label{section:LinearGPPA}
 
 In this subsection, we consider the linear convergence of generalized proximal point algorithms.
 
 \cref{theorem:JckAMetricSub} shows a $Q$-linear convergence result of generalized proximal point algorithms. 
 \begin{theorem} \label{theorem:JckAMetricSub}
 	Let $\bar{x} \in \zer A$.  Suppose that $A$ is metrically subregular at $\bar{x}$ for $0 \in A\bar{x}$, i.e., 
 	\begin{align*}
 	(\exists \kappa >0) (\exists \delta >0) (\forall x \in B[\bar{x}; \delta]) \quad \dist \lr{x, A^{-1}0} \leq \kappa \dist \lr{0, Ax}.
 	\end{align*}
 	Let $\lr{\varepsilon_{k} }_{k \in \mathbb{N}}$ be in $\mathbb{R}_{+}$ such that  $\eta_{k}  \varepsilon_{k} \to 0$.
 	Suppose that $(\forall k \in \mathbb{N})$ $\norm{e_{k}} \leq \varepsilon_{k} \norm{x_{k} -x_{k+1}}$ and $\lambda_{k} \in \left]0,2\right[\,$, 
 	that $\J_{c_{k} A}x_{k} \to \bar{x}$, 
 	and that $c:=\liminf_{k \to \infty} c_{k} > 0$ and $0 < \underline{\lambda} :=\liminf_{k \to \infty} \lambda_{k} \leq \overline{\lambda} :=\limsup_{k \to \infty} \lambda_{k} < 2$. 
 	Set 
 	\begin{align*}
 	(\forall k \in \mathbb{N}) \quad \rho_{k} :=  \max \left\{ \lr{1-\frac{\lambda_{k} }{ \frac{\kappa}{c_{k} }+1}}^{2}, \lr{ 1 - \lambda_{k}  \lr{2-\lambda_{k} } \frac{1}{1+\frac{\kappa^{2} }{c_{k}^{2} }} }   \right\}^{\frac{1}{2}}.  
 	\end{align*}  
 	Then the following statements hold.
 	\begin{enumerate}
 		\item \label{theorem:JckAMetricSub:dist} For every $k$ large enough, we have that $ \rho_{k} \in  \left]0,1\right[\,$ and that
 		\begin{align*}
 		\norm{x_{k+1}- \Pro_{\zer A} \lr{ \J_{ c_{k} A} x_{k} } }  \leq   \frac{\rho_{k} + \eta_{k} \varepsilon_{k}  }{1  - \eta_{k} \varepsilon_{k} }   \norm{x_{k}- \Pro_{\zer A} \lr{ \J_{ c_{k} A}x_{k}}}.
 		\end{align*}
 		Moreover,
 		there exist  $\mu \in \left[0,1\right[\,$ and $K \in \mathbb{N}$ such that 
 		\begin{align*}
 	\lr{\forall k \geq K} \quad 	\norm{x_{k+1}- \Pro_{\zer A} \lr{ \J_{ c_{k} A} x_{k} } }   \leq   \mu  \norm{x_{k}- \Pro_{\zer A} \lr{ \J_{ c_{k} A}x_{k}}}  
 		  \leq  \mu^{k-K+1}  \norm{x_{K}- \Pro_{\zer A} \lr{ \J_{ c_{K} A}x_{K}}}.
 		\end{align*}
 		\item \label{theorem:JckAMetricSub:norm} Suppose that $ \zer A = \{ \bar{x}  \}$. Then  for every $k$ large enough,
 		\begin{align} \label{eq:theorem:JckAMetricSub:LEQ}
 		\norm{x_{k+1}-  \bar{x}}  \leq \frac{\rho_{k} + \eta_{k} \varepsilon_{k} }{1  - \eta_{k} \varepsilon_{k}  } \norm{x_{k} -   \bar{x} }.
 		\end{align}
 		Moreover,
 		there exist $\mu \in \left[0,1\right[\,$ and $K \in \mathbb{N}$ such that 
 		\begin{align*}
 		(\forall k \geq K) \quad	\norm{x_{k+1}-  \bar{x}}  \leq \mu \norm{x_{k} -   \bar{x} }  \leq \mu^{k-K+1} \norm{x_{K} -   \bar{x} }.
 		\end{align*}
 	\end{enumerate}
 \end{theorem}
 
 \begin{proof}
 	\cref{theorem:JckAMetricSub:dist}: 
 	Inasmuch as  $\J_{c_{k} A}x_{k} \to \bar{x}  $ and $\eta_{k}  \varepsilon_{k} \to 0 $, there exists $K_{1} \in \mathbb{N}$ such that 
 	\begin{align} \label{eq:theorem:JckAMetricSub:JckB}
 	(\forall k \geq K_{1}) \quad \J_{c_{k} A}x_{k} \in B[\bar{x}; \delta] \quad \text{and} \quad \eta_{k} \varepsilon_{k}   \in \left[0, \frac{1}{2}\right].
 	\end{align}
 	For every $k  \geq K_{1}$, applying 	\cref{theorem:MetrSubregOptimalBoundszx}\cref{theorem:MetrSubregOptimalBoundszx:MetricSub:rho}$\&$\cref{theorem:MetrSubregOptimalBoundszx:MetricSub:zx}      with $x=x_{k}$, $z_{x}=x_{k+1}$, $\gamma =c_{k}$, $\lambda=\lambda_{k}$, $\eta =\eta_{k}$, $e=e_{k}$, and  $\varepsilon=\varepsilon_{k}$, and employing \cref{eq:theorem:JckAMetricSub:JckB}, we get  $ \rho_{k} \in  \left]0,1\right[\,$  and  
 	\begin{align} \label{eq:theorem:JckAMetricSub:xk+1LEQ}
 	\norm{x_{k+1}- \Pro_{\zer A} \lr{ \J_{ c_{k} A} x_{k} } }  \leq   \frac{\rho_{k} + \eta_{k} \varepsilon_{k} }{1  - \eta_{k}\varepsilon_{k} }   \norm{x_{k}- \Pro_{\zer A} \lr{ \J_{ c_{k} A}x_{k}}}.
 	\end{align}

 	In addition, because $c =\liminf_{k \to \infty} c_{k} > 0$ and $0 < \underline{\lambda}  =\liminf_{k \to \infty} \lambda_{k} \leq \overline{\lambda}  =\limsup_{k \to \infty} \lambda_{k} < 2$, we observe that $\limsup_{k \to \infty} \abs{ 1-\frac{\lambda_{k} }{ \frac{\kappa}{c_{k} }+1} } \leq \max \left\{ \abs{1 -  \frac{\underline{\lambda}}{ \frac{\kappa}{ c}+1 }} , \abs{\overline{\lambda} -1} \right\}$ and that 
 	\begin{align*}
   \rho := \limsup_{k \to \infty} \rho_{k} \leq \max \left\{  \abs{1 -  \frac{\underline{\lambda}}{ \frac{\kappa}{ c}+1 }} , \abs{\overline{\lambda} -1},  \lr{ 1 -\underline{\lambda} \lr{2-\overline{\lambda}} \frac{1}{ 1 + \frac{\kappa^{2}}{c^{2}}}}^{\frac{1}{2}}    \right\} <1.
 	\end{align*}
 	
 	Since $\eta_{k} \varepsilon_{k} \to 0 $,
 	we have that $	\limsup_{k \to \infty} \frac{\rho_{k} + \eta_{k}  \varepsilon_{k} }{1  - \eta_{k} \varepsilon_{k} } = \rho  <1$,
 	which necessitates that there exist  $K \geq K_{1}$ and $\mu \in \left]\rho,1\right[\,$ such that 
 	\begin{align*}
 	(\forall k \geq K) \quad \frac{\rho_{k} + \eta_{k}\varepsilon_{k} }{1  - \eta_{k} \varepsilon_{k} } \leq \mu.
 	\end{align*}
 	This combined with \cref{eq:theorem:JckAMetricSub:xk+1LEQ} deduces the last assertion in \cref{theorem:JckAMetricSub:dist}. 
 	
 	\cref{theorem:JckAMetricSub:norm}:  Notice that the assumption  $ \zer A = \{ \bar{x}  \}$ forces that 
 	\begin{align*}
 	\lr{\forall k \in \mathbb{N}} \quad  \Pro_{\zer A} \lr{ \J_{ c_{k} A} x_{k} } \equiv  \bar{x}.
 	\end{align*}
 	Therefore, \cref{theorem:JckAMetricSub:norm}   is immediate from \cref{theorem:JckAMetricSub:dist}.
 \end{proof}
 
 \begin{remark} \label{remark:theorem:JckAMetricSub}
 	We uphold the assumption  and notation of \cref{theorem:JckAMetricSub}. Taking \cref{corollary:lambda=1} into account, we observe that if $(\forall k \in \mathbb{N})$ $\lambda_{k} \equiv 1$, then 
 	\begin{align*}
 	(\forall k \in \mathbb{N}) \quad \rho_{k} = \lr{ 1 -  \frac{1}{1+\frac{\kappa^{2} }{c_{k}^{2} }} }^{\frac{1}{2}} =\frac{1}{ \sqrt{1 + \frac{c_{k}^{2}}{\kappa^{2}} } }.
 	\end{align*}
 	Note that this   convergence rate is consistent with that of \cref{theorem:exactPAlinear}\cref{theorem:exactPAlinear:1:leq}.
 \end{remark}

 \begin{proposition} \label{prop:JckAMetricSub}
 	Suppose that $\mathcal{H} = \mathbb{R}^{n}$.	Suppose that $ \bar{x}  \in \zer A$ and that $A$ is metrically subregular at $\bar{x}$ for $0 \in A\bar{x}$, i.e., 
 	\begin{align*} 
 	(\exists \kappa >0) (\exists \delta >0) (\forall x \in B[\bar{x}; \delta]) \quad \dist \lr{x, A^{-1}0} \leq \kappa \dist \lr{0, Ax}.
 	\end{align*}
 	Suppose that $\sum_{k \in \mathbb{N}} \eta_{k} \norm{e_{k}}  < \infty $. Let $\lr{\varepsilon_{k} }_{k \in \mathbb{N}}$ be in $\mathbb{R}_{+}$ such that  $\eta_{k}  \varepsilon_{k} \to 0$. 	Suppose that $(\forall k \in \mathbb{N})$ $\norm{e_{k}} \leq \varepsilon_{k} \norm{x_{k} -x_{k+1}}$ and $\lambda_{k} \in \left]0,2\right[\,$, and that $ \liminf_{k \to \infty} c_{k} > 0$ and $0 <  \liminf_{k \to \infty} \lambda_{k} \leq  \limsup_{k \to \infty} \lambda_{k} < 2$. 
 	Set 
 	\begin{align*}
 	(\forall k \in \mathbb{N}) \quad \rho_{k} :=  \max \left\{ \lr{1-\frac{\lambda_{k} }{ \frac{\kappa}{c_{k} }+1}}^{2}, \lr{ 1 - \lambda_{k}  \lr{2-\lambda_{k} } \frac{1}{1+\frac{\kappa^{2} }{c_{k}^{2} }} }   \right\}^{\frac{1}{2}}.
 	\end{align*}   
 	Then the following statements hold.
 	\begin{enumerate}	
 		\item   For every $k$ large enough, we have that 
 		\begin{align*}
 		\norm{x_{k+1}- \Pro_{\zer A} \lr{ \J_{ c_{k} A} x_{k} } }  \leq   \frac{\rho_{k} + \eta_{k} \varepsilon_{k} }{1  - \eta_{k} \varepsilon_{k}}   \norm{x_{k}- \Pro_{\zer A} \lr{ \J_{ c_{k} A}x_{k}}}.
 		\end{align*}
 		Moreover,
 		there exist  $\mu \in \left[0,1\right[\,$ and $K \in \mathbb{N}$ such that
 		\begin{align*}
 		\lr{\forall k \geq K} \quad 	\norm{x_{k+1}- \Pro_{\zer A} \lr{ \J_{ c_{k} A} x_{k} } }   \leq   \mu  \norm{x_{k}- \Pro_{\zer A} \lr{ \J_{ c_{k} A}x_{k}}}  
 		  \leq  \mu^{k-K +1}  \norm{x_{K}- \Pro_{\zer A} \lr{ \J_{ c_{K} A}x_{K}}}. 
 		\end{align*}
 		\item  Suppose that $ \zer A = \{ \bar{x}  \}$. Then  for every $k$ large enough,
 		\begin{align*}  
 		\norm{x_{k+1}-  \bar{x}}  \leq \frac{\rho_{k} + \eta_{k}\varepsilon_{k} }{1  - \eta_{k} \varepsilon_{k} } \norm{x_{k} -   \bar{x} }.
 		\end{align*}
 		Moreover,
 		there exist  $\mu \in \left[0,1\right[\,$ and $K \in \mathbb{N}$ such that 
 		\begin{align*}
 		(\forall k \geq K) \quad	\norm{x_{k+1}-  \bar{x}}  \leq \mu \norm{x_{k} -   \bar{x} }  \leq \mu^{k-K+1} \norm{x_{K} -   \bar{x} }.
 		\end{align*}
 	\end{enumerate}
 \end{proposition}
 
 \begin{proof}
 	Because $(c_{k})_{k \in \mathbb{N}}$ is in $\mathbb{R}_{++}$, 
 	 it is not difficult to verify that $	\liminf_{k \to \infty} c_{k} > 0 \Leftrightarrow \inf_{k \in \mathbb{N}}c_{k}>0$.
 	Moreover,	bearing $\mathcal{H} = \mathbb{R}^{n}$, $\sum_{k \in \mathbb{N}} \eta_{k} \norm{e_{k}}  < \infty $, and $0 <  \liminf_{k \to \infty} \lambda_{k} \leq  \limsup_{k \to \infty} \lambda_{k} < 2$ in mind, and employing
 	\cref{corolary:Axkweakconverge}, we know that  $	x_{k} \to \bar{x}$. Notice that, via \cref{lemma:JGammaAFix}, $(\forall k \in \mathbb{N})$ $\norm{ \J_{c_{k} A}x_{k} - \bar{x}} \leq \norm{ x_{k} - \bar{x} }$. These results entail that	
 	\begin{align*}
 	\J_{c_{k} A}x_{k} \to \bar{x}.
 	\end{align*}
 	Therefore,  in view of \cref{theorem:JckAMetricSub}, we obtain the required results. 
 \end{proof}

 \begin{theorem} \label{theorem:JckALipschitzLinear}
 	Suppose that $A^{-1}$ is Lipschitz continuous at $0$ with modulus $\alpha >0$, i.e., $A^{-1}(0) =\{\bar{x}\}$ and there exists $\tau >0$  such that 
 	\begin{align*} 
 	\lr{ \forall (w,x) \in \gra A^{-1} \text{ with } w \in B[0;\tau]} \quad \norm{x -\bar{x}} \leq \alpha \norm{w}.
 	\end{align*}
 	Suppose that 
 	$\frac{1}{c_{k}} \lr{x_{k} -\J_{c_{k} A}x_{k}} \to 0$. Let $\lr{\varepsilon_{k} }_{k \in \mathbb{N}}$ be in $\mathbb{R}_{+}$ such that  $\eta_{k}  \varepsilon_{k} \to 0$. 	Suppose that $(\forall k \in \mathbb{N})$ $\norm{e_{k}} \leq \varepsilon_{k} \norm{x_{k} -x_{k+1}}$ and $\lambda_{k} \in \left]0,2\right[\,$,  
 	and that $ \liminf_{k \to \infty} c_{k} > 0$  and $0 <  \liminf_{k \to \infty} \lambda_{k} \leq  \limsup_{k \to \infty} \lambda_{k} < 2$.
 	Set 
 	\begin{align*}
 	(\forall k \in \mathbb{N}) \quad \rho_{k} :=  \max \left\{ \lr{1-\frac{\lambda_{k} }{ \frac{\alpha}{c_{k} }+1}}^{2}, \lr{ 1 - \lambda_{k}  \lr{2-\lambda_{k} } \frac{1}{1+\frac{\alpha^{2} }{c_{k}^{2} }} }   \right\}^{\frac{1}{2}}. 
 	\end{align*}
 	Then the following statements hold.
 	\begin{enumerate}	
 		\item  \label{theorem:JckALipschitzLinear:leq} For every $k$ large enough,  we have that $ \rho_{k} \in  \left]0,1\right[$ and that
 		\begin{align*}
 		\norm{x_{k+1}-  \bar{x}}  \leq \frac{\rho_{k} + \eta_{k} \varepsilon_{k} }{1  - \eta_{k} \varepsilon_{k}  } \norm{x_{k} -   \bar{x} }.
 		\end{align*}
 		\item 	 \label{theorem:JckALipschitzLinear:mu} There exist $\mu \in \left[0,1\right[\,$ and $K \in \mathbb{N}$ such that 
 		\begin{align}  \label{eq:theorem:JckALipschitzLinear:mu} 
 		(\forall k \geq K) \quad	\norm{x_{k+1}-  \bar{x}}  \leq \mu \norm{x_{k} -   \bar{x} } \leq \mu^{k-K+1} \norm{x_{K} -   \bar{x} }.
 		\end{align}
 	\end{enumerate}
 \end{theorem}
 
 \begin{proof}
 	\cref{theorem:JckALipschitzLinear:leq}:  In view of
 	$\frac{1}{c_{k}} \lr{x_{k} -\J_{c_{k} A}x_{k}} \to 0$ and $\eta_{k}  \varepsilon_{k} \to 0$,   	
 	there exists $K_{1} \in \mathbb{N}$ such that 
 	\begin{align} \label{eq:theorem:JckALipschitzLinear:fracJckB}
 	(\forall k \geq K_{1}) \quad \frac{1}{c_{k}} \lr{x_{k} -\J_{c_{k} A}x_{k}}  \in B[0;\tau] \quad \text{and} \quad \eta_{k}  \varepsilon_{k}  \in \left[0, \frac{1}{2}\right].
 	\end{align}
 	For every $k \geq K_{1}$, applying 	\cref{theorem:MetrSubregOptimalBoundszx}\cref{theorem:MetrSubregOptimalBoundszx:Lipschitz:rho}$\&$\cref{theorem:MetrSubregOptimalBoundszx:Lipschitz:zx}   with $x=x_{k}$, $z_{x}=x_{k+1}$, $\gamma =c_{k}$, $\lambda=\lambda_{k}$, $\eta =\eta_{k}$, $e=e_{k}$, and  $\varepsilon=\varepsilon_{k}$, 
 	and employing \cref{eq:theorem:JckALipschitzLinear:fracJckB}, we derive $ \rho_{k} \in  \left]0,1\right[$   and  
 	\begin{align} \label{eq:theorem:JckALipschitzLinear:xk+1LEQ}
 		\norm{x_{k+1}-  \bar{x}} \leq \frac{\rho_{k} + \eta_{k} \varepsilon_{k} }{1  - \eta_{k} \varepsilon_{k} } \norm{x_{k} -   \bar{x} }.
 	\end{align}
 	
 	\cref{theorem:JckALipschitzLinear:mu}: Similarly with the proof of the last part of  \cref{theorem:JckAMetricSub}\cref{theorem:JckAMetricSub:dist}, there exist $K \geq K_{1}$ and $\mu \in \left[0,1\right[\,$ such that 
 	\begin{align*}
 	(\forall k \geq K) \quad \frac{\rho_{k} + \eta_{k} \varepsilon_{k} }{1  - \eta_{k} \varepsilon_{k} } \leq \mu,
 	\end{align*}
 	which, combined with \cref{eq:theorem:JckALipschitzLinear:xk+1LEQ}, guarantees \cref{eq:theorem:JckALipschitzLinear:mu}. 
 \end{proof}

 \begin{proposition} \label{prop:JckALipschitzLinear}
 	Suppose that   $A^{-1}$ is Lipschitz continuous at $0$ with modulus $\alpha >0$, i.e., $A^{-1}(0) =\{\bar{x}\}$ and there exists $\tau >0$  such that 
 	\begin{align*}   
 	\lr{ \forall (w,x) \in \gra A^{-1} \text{ with } w \in B[0;\tau]} \quad \norm{x -\bar{x}} \leq \alpha \norm{w}.
 	\end{align*}
 	Suppose that  $\sum_{k \in \mathbb{N}} \eta_{k} \norm{e_{k}}  < \infty $.   Let $\lr{\varepsilon_{k} }_{k \in \mathbb{N}}$ be in $\mathbb{R}_{+}$ such that  $\eta_{k}  \varepsilon_{k} \to 0$. 	Suppose that $(\forall k \in \mathbb{N})$ $\norm{e_{k}} \leq \varepsilon_{k} \norm{x_{k} -x_{k+1}}$ and $\lambda_{k} \in \left]0,2\right[\,$,  
 	and that $ \liminf_{k \to \infty} c_{k} > 0$  and $0 <  \liminf_{k \to \infty} \lambda_{k} \leq  \limsup_{k \to \infty} \lambda_{k} < 2$.
 	Set 
 	\begin{align*}
 	(\forall k \in \mathbb{N}) \quad \rho_{k} :=  \max \left\{ \lr{1-\frac{\lambda_{k} }{ \frac{\alpha}{c_{k} }+1}}^{2}, \lr{ 1 - \lambda_{k}  \lr{2-\lambda_{k} } \frac{1}{1+\frac{\alpha^{2} }{c_{k}^{2} }} }   \right\}^{\frac{1}{2}}. 
 	\end{align*}
 	Then the following hold. 
 	\begin{enumerate}	
 		\item  \label{prop:JckALipschitzLinear:leq}    For every $k$ large enough, we have that
 		\begin{align} \label{eq:prop:JckALipschitzLinear:xk+1}
 		\norm{x_{k+1}-  \bar{x}}  \leq \frac{\rho_{k} + \eta_{k} \varepsilon_{k} }{1  - \eta_{k} \varepsilon_{k} } \norm{x_{k} -   \bar{x} }.
 		\end{align}
 		\item 		\label{prop:JckALipschitzLinear:mu}   There exist $\mu \in \left[0,1\right[\,$ and $K \in \mathbb{N}$ such that 
 		\begin{align*}
 		(\forall k \geq K) \quad	\norm{x_{k+1}-  \bar{x}} \leq \mu \norm{x_{k} -   \bar{x} } \leq \mu^{k-K+1} \norm{x_{K} -   \bar{x} }.
 		\end{align*}
 	\end{enumerate} 
 \end{proposition}
 
 \begin{proof}
 	Similarly with the proof of \cref{prop:JckAMetricSub}, the assumptions  $ \liminf_{k \to \infty} c_{k} > 0$ and $(\forall k \in \mathbb{N})$ $c_{k} >0$ ensure that  $\inf_{k \in \mathbb{N}} c_{k} >0$.
 	Then, via \cref{corollary:ckxkxk+1to0}\cref{corollary:ckxkxk+1to0:Jck}, we establish that 
 	$ \frac{1}{c_{k}} \lr{x_{k} -\J_{c_{k} A}x_{k}}  \to 0$.
 	
 	Hence, as a consequence of \cref{theorem:JckALipschitzLinear}, we obtain the required results.
 \end{proof}
 
 \begin{remark}
 	We uphold the notation used in \cref{prop:JckALipschitzLinear}.
 	\begin{enumerate} 
 		
 		\item \cref{prop:JckALipschitzLinear} improves \cite[Theorem~4.7]{TaoYuan2018} from the following two aspects. 
 		\begin{itemize}
 			\item Note that  $0 <  \inf_{k \in \mathbb{N}} \lambda_{k} \leq \sup_{k \in \mathbb{N}} \lambda_{k} < 2 $ if and only if $(\forall k \in \mathbb{N})$ $  \lambda_{k} \in \left]0,2\right[$  and  $0 <  \liminf_{k \to \infty} \lambda_{k} \leq  \limsup_{k \to \infty} \lambda_{k} < 2$.
 			We see clearly that the constant $\lambda \in \left]0,2\right[$ in \cite[Theorem~4.7]{TaoYuan2018}  is replaced by a sequence $(\lambda_{k})_{k \in \mathbb{N}}$ satisfying $0 <  \inf_{k \in \mathbb{N}} \lambda_{k} \leq \sup_{k \in \mathbb{N}} \lambda_{k} < 2$ in \cref{prop:JckALipschitzLinear}.

 			\item 
 			Notice that the linear convergence result provided in \cite[Theorem~4.7]{TaoYuan2018} is essentially  that for every $k$ large enough, $\norm{x_{k+1}-  \bar{x}}  \leq \frac{\zeta_{k} + \eta_{k} \varepsilon_{k} }{1  - \eta_{k} \varepsilon_{k} } \norm{x_{k} -   \bar{x} }$,
 			where 
 			\begin{align*}
 	(\forall k \in \mathbb{N}) \quad		\zeta_{k} :=  \max \left\{ 
 			\lr{ 1 - \lambda    \frac{1}{1+\frac{\alpha^{2} }{c_{k}^{2} }} }    , \lr{ 1 - \lambda  \lr{2-\lambda } \frac{1}{1+\frac{\alpha^{2} }{c_{k}^{2} }} }   \right\}^{\frac{1}{2}} \text{ with } \lambda \in \left]0,2\right[\,. 
 			\end{align*}
 			Bearing  \cref{remark:theorem:MetrSubregOptimalBounds} in mind, we conclude that the convergence rate given in \cref{eq:prop:JckALipschitzLinear:xk+1} of our  \cref{prop:JckALipschitzLinear}\cref{prop:JckALipschitzLinear:leq}  is better than the corresponding rate in  \cite[Theorem~4.7]{TaoYuan2018}.
 		\end{itemize}
 		\item 	Taking \cref{corollary:lambda=1} into account, we observe that if $(\forall k \in \mathbb{N})$ $\lambda_{k} \equiv 1$, then $	(\forall k \in \mathbb{N}) $ $ \rho_{k} = \lr{ 1 -  \frac{1}{1+\frac{\alpha^{2} }{c_{k}^{2} }} }^{\frac{1}{2}} =\frac{1}{ \sqrt{1 + \frac{c_{k}^{2}}{\alpha^{2}} } } = \frac{\alpha}{ \lr{\alpha^{2} + c_{k}^{2}}^{\frac{1}{2}} }$.
 	Therefore, we see that  \cref{prop:JckALipschitzLinear}  extends \cite[Theorem~2]{Rockafellar1976} from $\lambda \equiv 1$ to  a sequence $(\lambda_{k})_{k \in \mathbb{N}}$ satisfying $0 <  \inf_{k \in \mathbb{N}} \lambda_{k} \leq \sup_{k \in \mathbb{N}} \lambda_{k} < 2$.
 	\end{enumerate}
 \end{remark}
 
 \cref{theorem:MetriSubregLinearNosingleton} below provides $R$-linear convergence results on generalized proximal point algorithms. It
 is an application of   \cref{theorem:Tkdistlinear} which shows a $R$-linear convergence result on the  non-stationary Krasnosel'ski\v{\i}-Mann iterations.
 \begin{theorem} \label{theorem:MetriSubregLinearNosingleton}
Let $\bar{x} \in \zer A$. 	Suppose that  $A$ is metrically subregular at $\bar{x} $ for $0 \in A\bar{x}$, i.e., 
 	\begin{align} \label{eq:theorem:MetriSubregLinearNosingleton:MetricSub} 
 	(\exists \kappa >0) (\exists \delta >0) (\forall x \in B[\bar{x}; \delta]) \quad \dist \lr{x, A^{-1}0} \leq \kappa \dist \lr{0, Ax}.
 	\end{align}
 	Suppose that $c:=\inf_{k \in \mathbb{N}} c_{k} >0$ and 
 	that $  \sum_{k \in \mathbb{N}} \eta_{k}\norm{e_{k}} < \delta$.
 	Let $0 < \tau \leq  \delta - \sum_{k \in \mathbb{N}} \eta_{k}\norm{e_{k}}  $ and let $x_{0} \in B[\bar{x}; \tau]$.
 	Define 
 	\begin{align*}
 \lr{\forall k \in \mathbb{N}} \quad 	y_{k}:= (1-\lambda_{k})x_{k} +\lambda_{k} T_{k}x_{k} \text{ and } \varepsilon_{k} :=  \eta_{k} \norm{e_{k}} \lr{2\norm{y_{k} -\Pro_{\zer A}x_{k}} +\eta_{k} \norm{e_{k}}}.	
 	\end{align*} 
 	Set $(\forall k \in \mathbb{N})$ 
 	$\gamma_{k} := \lr{1 + \frac{\kappa}{c_{k} }}$,   $\beta_{k}:= \frac{\lambda_{k} \lr{2 -\lambda_{k}} }{\gamma_{k}^{2}}$, and $\rho_{k} := 1-\beta_{k}	$.	Denote by $ \rho  := \sup_{k \in \mathbb{N}} \rho_{k}$, $\underline{\lambda}:=\inf_{k \in \mathbb{N}} \lambda_{k} $, and $ \overline{\lambda}:=\sup_{k \in \mathbb{N}} \lambda_{k} $. 
 	Then the following statements hold.
 	\begin{enumerate}
 		\item \label{theorem:MetriSubregLinearNosingleton:RhoLeq}  $(\forall k \in \mathbb{N})$ $\rho_{k} \in \left[0,1\right]$ and  $\dist^{2} \lr{x_{k+1}, \zer A} \leq \rho_{k} \dist^{2} \lr{x_{k}, \zer A} +\varepsilon_{k}$.
 		\item  \label{theorem:MetriSubregLinearNosingleton:leq} $	(\forall k \in \mathbb{N}) $ $ \dist^{2} \lr{x_{k+1}, \zer A} \leq  \lr{ \prod^{k}_{i=0} \rho_{i}  } \dist^{2} \lr{x_{0}, \zer A} + \sum^{k}_{i=0} \lr{ \prod^{k}_{j=i+1} \rho_{j} } \varepsilon_{i}.$
 		\item  \label{theorem:MetriSubregLinearNosingleton:lambda} Suppose that $0 < \underline{\lambda} \leq \overline{\lambda} <2 $. Then the following hold.  
 		\begin{enumerate}
 			\item  \label{theorem:MetriSubregLinearNosingleton:barrho}  $0 \leq  \rho   \leq 1 -  \frac{ \underline{\lambda} \lr{2 - \overline{\lambda}}}{ \lr{1 + \frac{\kappa}{c}}^{2} } < 1$.
 			\item   \label{theorem:MetriSubregLinearNosingleton:rhok} $(\forall k \in \mathbb{N}) $ $\dist^{2} \lr{x_{k+1},  \zer A} \leq  \rho^{k} \lr{    \rho  \dist^{2} \lr{x_{0},  \zer A} + \sum^{k}_{i=0}  \frac{ \varepsilon_{i} }{  \rho^{i} } }.  $ Consequently, if $\sum_{k \in \mathbb{N}}  \frac{ \varepsilon_{k} }{ \rho^{k} } < \infty$, then $\lr{\dist^{2} \lr{x_{k},  \zer A}}_{ k \in \mathbb{N}}$ converges $R$-linearly to $0$.  		
 		\end{enumerate}
 		\item  \label{theorem:MetriSubregLinearNosingleton:ek=0} Suppose that $0 < \underline{\lambda} \leq \overline{\lambda} <2 $ and that  $(\forall k \in \mathbb{N})$ $e_{k} \equiv 0$ and $\eta_{k} \equiv 0$. Then the following hold.  
 		\begin{enumerate}
 			\item  $(\forall k \in \mathbb{N})$ $\dist  \lr{x_{k+1},  \zer A} \leq \rho_{k}^{\frac{1}{2}} \dist  \lr{x_{k},  \zer A}  $.
 			\item There exists a point $\hat{x} \in \zer A$ such that 
 			\begin{align*}
 			(\forall k \in \mathbb{N}) \quad \norm{x_{k} -\hat{x}} \leq 2  \rho^{\frac{k}{2}} \dist \lr{x_{0}, \zer A}.
 			\end{align*}
 			Consequently, $\lr{x_{k}}_{k \in \mathbb{N}}$ converges $R$-linearly to a point $\hat{x} \in  \zer A$.
 		\end{enumerate} 
 	\end{enumerate}
 \end{theorem}
 
 \begin{proof}
 	In view of \cref{cor:fact:FixJcAzerA},  $(\forall k \in \mathbb{N})$ $\J_{c_{k} A}$ is $\frac{1}{2}$-averaged operator and 
 	\begin{align*}
 	(\forall k \in \mathbb{N}) \quad \lr{\Id -\J_{c_{k} A}  }^{-1}0 = \Fix \J_{c_{k} }A =\zer A.
 	\end{align*}
 	
 	Moreover, employing \cref{eq:theorem:MetriSubregLinearNosingleton:MetricSub} and for every $k \in \mathbb{N}$, applying \cref{lemma:metricallysubregularEQ} with $\gamma=c_{k}$, we know that $(\forall k \in \mathbb{N})$ $  \Id -\J_{c_{k} A}  $ is metrically subregular at $\bar{x}$ for $0 = \lr{\Id -\J_{c_{k} A} } \bar{x}$; more precisely, 	
 	\begin{align} \label{eq:theorem:MetriSubregLinearNosingleton:Id-JckA}
 \lr{\forall x \in B[\bar{x}; \delta]} \quad 	\dist \lr{x, \Fix \J_{c_{k} }A} \leq \lr{1 + \frac{\kappa}{c_{k} }}   \norm{x-\J_{c_{k}  A} x}.
 	\end{align}

 	\cref{theorem:MetriSubregLinearNosingleton:RhoLeq}$\&$\cref{theorem:MetriSubregLinearNosingleton:leq}: Inasmuch as $(\forall k \in \mathbb{N})$ $\lambda_{k} \in \left[0,2\right]$ and $\gamma_{k} >1$, we have that  $(\forall k \in \mathbb{N})$ $ \lambda_{k} \lr{2 -\lambda_{k}} \in \left[0,1\right]$ and $\beta_{k} = \frac{\lambda_{k} \lr{2 -\lambda_{k}} }{\gamma_{k}^{2}} \in \left[0,1 \right]$. Hence, $(\forall k \in \mathbb{N})$ $\rho_{k} \in \left[0,1\right]$.
 	
 		Apply	\cref{theorem:Tkdistlinear}\cref{theorem:Tkdistlinear:rhok}$\&$\cref{theorem:Tkdistlinear:prod}   with $(\forall k \in \mathbb{N})$ $T_{k} = \J_{c_{k} A}$, $C= \zer A$, $\gamma_{k} = \lr{1 + \frac{\kappa}{c_{k} }}$, $\delta_{k} \equiv \delta$, and $\alpha_{k} \equiv \frac{1}{2}$ to derive results  in \cref{theorem:MetriSubregLinearNosingleton:RhoLeq}$\&$\cref{theorem:MetriSubregLinearNosingleton:leq}.

 	\cref{theorem:MetriSubregLinearNosingleton:barrho}:	Note that 
 	\begin{align*}
 	&\sup_{k \in \mathbb{N}}   \gamma_{k} =	\sup_{k \in \mathbb{N}}  \lr{1 + \frac{\kappa}{c_{k} } } =   1 + \frac{\kappa}{ \inf_{k \in \mathbb{N}} c_{k} } = 1 + \frac{\kappa}{c}\\
 	\Rightarrow & \rho =  \sup_{k \in \mathbb{N}} \rho_{k} = 1- \inf_{k \in \mathbb{N}}  \beta_{k}  \leq 1 -  \frac{ \underline{\lambda} \lr{2 - \overline{\lambda}}}{ \lr{1 + \frac{\kappa}{c}}^{2} } \in \left]0,1\right[\,.
 	\end{align*}
 	Hence, \cref{theorem:MetriSubregLinearNosingleton:barrho} is true. 
 	
 	\cref{theorem:MetriSubregLinearNosingleton:rhok}: Employing \cref{theorem:MetriSubregLinearNosingleton:barrho} above and applying \cref{theorem:Tkdistlinear}\cref{theorem:Tkdistlinear:prod:linera}  with $(\forall k \in \mathbb{N})$ $T_{k} = \J_{c_{k} A}$, $C= \zer A$, $\gamma_{k} = \lr{1 + \frac{\kappa}{c_{k} }}$, $\delta_{k} \equiv \delta$, and $\alpha_{k} \equiv \frac{1}{2}$, we establish the required  results in \cref{theorem:MetriSubregLinearNosingleton:rhok}.
 	
 	\cref{theorem:MetriSubregLinearNosingleton:ek=0}: Taking  \cref{theorem:MetriSubregLinearNosingleton:barrho} above into account and applying  \cref{Corollary:Tkdistlinear} with $(\forall k \in \mathbb{N})$ $T_{k} = \J_{c_{k} A}$, $C= \zer A$, $\gamma_{k} = \lr{1 + \frac{\kappa}{c_{k} }}$, $\delta_{k} \equiv \delta$, and $\alpha_{k} \equiv \frac{1}{2}$, we directly obtain the desired results in \cref{theorem:MetriSubregLinearNosingleton:ek=0}.
 \end{proof}

\section{Conclusion and Future Work} \label{section:ConclusionFutureWork}
In this work, we considered the metrical subregularity of set-valued operators, which is  a   popular assumption for the linear convergence of optimization algorithms. We
 also provided an $R$-linear convergence result on the  non-stationary Krasnosel'ski\v{\i}-Mann iterations. Because the generalized proximal point algorithm is a special instance of the non-stationary Krasnosel'ski\v{\i}-Mann iterations, the result on the  non-stationary Krasnosel'ski\v{\i}-Mann iterations was applied to the generalized proximal point algorithm. In addition, we showed some $Q$-linear convergence results on the generalized proximal point algorithm under the assumption that the associated monotone operator is metrically subregular or that the inverse of the monotone operator is Lipschitz continuous with a positive modulus.

	Given a maximally monotone operator $A :\mathcal{H} \to 2^{\mathcal{H}}$  with   $\zer A =\{\bar{x}\}$, as stated in \cref{question:MetriSub} and \cref{remark:theorem:MetrSubregOptimalBoundszx}\cref{remark:theorem:MetrSubregOptimalBoundszx:better}, it is interesting to know if the following two statements are equivalent. 
	\begin{enumerate}
		\item \label{FW:MetricSub}  $A$ is metrically subregular at $\bar{x}$ for $0 \in A \bar{x}$.
		\item \label{FW:LipschitzConti}  $A^{-1}$ is Lipschitz continuous at $0$ with a positive modulus.
	\end{enumerate} 
	We stated  in \cref{prop:LipschitzMetricSubregularity}  that \cref{FW:LipschitzConti} implies \cref{FW:MetricSub} and we also  found  in \cref{example:fmetricsub}  a specific $A: \mathbb{R} \to \mathbb{R}$ suggesting a positive answer for the equivalence. In the future, we shall try either proving  \cref{FW:MetricSub} $\Rightarrow$ \cref{FW:LipschitzConti} or finding a counterexample for it. 
	 In addition, as we presented in the introduction, many popular optimization algorithms are   instances of the generalized proximal point algorithm when the associated monotone operator is specified accordingly. We shall also apply our linear convergence results to some particular examples of the generalized proximal point algorithm.

%\section*{Acknowledgements}

\addcontentsline{toc}{section}{References}

\bibliographystyle{abbrv}

\end{document}